\documentclass{amsart}

\usepackage[bookmarksnumbered,colorlinks, linkcolor=blue, citecolor=red, pagebackref, bookmarks, breaklinks]{hyperref}

\usepackage[hyperpageref]{backref}
\usepackage{amsfonts,color}
\usepackage{latexsym,amssymb}
\usepackage{amsmath}
\usepackage[utf8]{inputenc}
\usepackage[none]{hyphenat}

\usepackage{verbatim} 

\newcommand*{\avint}{\mathop{\ooalign{$\int_{\mathbb{S}^n}$\cr$-$}}}

\begin{document}
\sloppy 

\title[Infinitely many sign-changing solutions of a critical fractional equation]{Infinitely many sign-changing solutions of a critical fractional equation}
\author{Emerson Abreu}
\address{Universidade Federal de Minas Gerais (UFMG), Departamento de Matem\'{a}tica, Caixa Postal 702, 30123-970, Belo Horizonte, MG, Brazil}
\email{eabreu@ufmg.br}
\author{Ezequiel Barbosa}
\address{Universidade Federal de Minas Gerais (UFMG), Departamento de Matem\'{a}tica, Caixa Postal 702, 30123-970, Belo Horizonte, MG, Brazil}
\email{ezequiel@mat.ufmg.br}
\author{Joel Cruz Ramirez}
\address{Universidade Federal de Minas Gerais (UFMG), Departamento de Matem\'{a}tica, Caixa Postal 702, 30123-970, Belo Horizonte, MG, Brazil}
\email{jols.math@gmail.com}
\thanks{This study was financed by the Brazilian agencies: Conselho Nacional de Desenvolvimento Cient\'{i}fico e Tecnol\'{o}gico (CNPq), Coordena\c c\~ao de
Aperfei\c coamento de Pessoal de N\'{i}vel Superior (CAPES), and Funda\c c\~ao de Amparo \`a Pesquisa do Estado de Minas Gerais (FAPEMIG)}

\date{\today}

\begin{abstract}
In this paper, we obtain non existence results of positive solutions, and also the existence of an unbounded sequence of solutions that changing sign for some critical problems involving conformally invariant operators on the standard unit sphere, and the fractional Laplacian operator in the Euclidean space. Our arguments are based on a reduction of the initial problem in the Euclidean space to an equivalent problem on the standard unit sphere and vice versa, what together to blow up arguments, a variant of Pohozaev's type identity, a refinement of regularity results for this type operators, and finally, by exploiting the symmetries of the sphere.
\end{abstract}

\maketitle
\numberwithin{equation}{section}
\newtheorem{theorem}{Theorem}[section]
\newtheorem{corollary}[theorem]{Corollary}
\newtheorem{remark}{Remark}[section]
\newtheorem{lemma}[theorem]{Lemma}
\newtheorem{definition}[theorem]{Definition}
\newtheorem{proposition}[theorem]{Proposition}
\newtheorem{example}[theorem]{Example}
\allowdisplaybreaks

\section{Introduction} 

	In recent years, a wide variety of problems involving nonlocal operators have been intensively studied by many researchers. Examples of these operators are the conformally invariant operators $P_k^g$  on Riemannian manifolds $(M^n,g)$ constructed by  Graham et al.\cite{GJM}, where $k$ is any positive integer if $n$ is odd, or $k \in \{1,.., n/2\}$ if $n$ is even. Moreover, for $n>2$, we have that $P_1^g$ is the well known conformal Laplacian $-\Delta_g + c(n)R_g$, where $\Delta_g$ is the Laplace-Beltrami operator, $c(n) = (n-2)/4(n-1)$, $P_1^g(1)=R_g$ is the scalar curvature of $M$. Also, for $n>4$, $P_2^g$ is the Paneitz operator and $P_2^g(1)$ is the $Q$-curvature.  
	
	Making use of a generalized Dirichlet to Neumann map, Graham and Zworski \cite{GZ} showed the existence of the conformally invariant operators $P^g_s$ of non-integer order $s\in(0,n/2)$ on the conformal infinity of asymptotically hyperbolic manifolds. Using this result, Chang and Gonz\'{a}lez \cite{CG} were able to define conformally invariant operators $P_s^g$ of non-integer order, by using the localization method of Caffarelli and Silvestre \cite{CS}. Thus, these  operators lead naturally to a fractional order curvature $R_s^g=P_s^g(1)$, which is called $s$-curvature. A typical example is the standard unit sphere $(\mathbb{S}^n, g_{\mathbb{S}^n})$ with $n>2$. Let $g$ be a representative in the conformal class $[g_{\mathbb{S}^n}] = \{\tilde{\rho} g_{\mathbb{S}^n} ; 0 < \tilde{\rho} \in C^{\infty}({\mathbb{S}^n} )\}$. By looking to $(\mathbb{S}^n, [g_{\mathbb{S}^n}])$ as the conformal infinity of the Poincar\'e ball $(\mathbb{B}^{n+1},g_{\mathbb{B}^{n+1}})$ and using Graham-Zworski's work, we obtain a family of conformally covariant (pseudo) differential operators $P_s^g$ for $g \in [g_{\mathbb{S}^n}]$ and $s\in(0,n/2)$. These operators satisfy the following conformal transformation relation (see also \cite{CG, GQ}):
\begin{equation}
\label{r1.1}
P_s^{g}(u)=\rho^{-\frac{n+2s}{n-2s}}P_s^{g_{\mathbb{S}^n}}(\rho u) \text{ for all } \rho, u \in C^{\infty}(\mathbb{S}^n) \text{ with } g=\rho^{\frac{4}{n-2s}} g_{\mathbb{S}^n}.
\end{equation}
Moreover, the $s$-curvature for $(\mathbb{S}^n, g)$ can be computed as
\begin{equation*}
R^g_s=\rho^{-\frac{n+2s}{n-2s}}P_s(\rho),
\end{equation*} 
where $P_s:=P_s^{g_{\mathbb{S}^n}}$ can be written as (see Beckner \cite{Bec}, Branson \cite{Bra}, and Morpurgo \cite{CM}):
\begin{equation*}
	P_s=\frac{\Gamma\left(B + \frac{1}{2}+s\right)}{\Gamma\left(B + \frac{1}{2}-s\right)},~ B=\sqrt{-\Delta_{g_{\mathbb{S}^n}} +\left(\frac{n-1}{2}\right)^2},
	\end{equation*}
	where $\Gamma$ is the Gamma function and $\Delta_{g_{\mathbb{S}^n}}$ is the Laplace-Beltrami operator on $(\mathbb{S}^n , g_{\mathbb{S}^n})$. Using stereographic projection we obtain a better view of $P_s$ on $\mathbb{R}^n$. Indeed, let $N$ be the north pole of $\mathbb{S}^n$, and $\mathcal{F}^{-1}: \mathbb{S}^n\backslash \{N\}\rightarrow\mathbb{R}^n$ the stereographic projection, which is the inverse of 
\begin{equation*}
\mathcal{F}:\mathbb{R}^n\rightarrow \mathbb{N}^n\backslash \{S\},~ y\mapsto \left(\frac{2y}{1+|y|^2},-\frac{1-|y|^2}{1+|y|^2} \right).
\end{equation*}
Then
	\begin{equation*}
	(P_s u)\circ \mathcal{F}=|J_{\mathcal{F}}|^{-\frac{n+2s}{2n}} (-\Delta)^s (|J_{\mathcal{F}}|^{\frac{n-2s}{2n}} (u\circ \mathcal{F}))~\text{ for all } u\in C^{\infty}(\mathbb{S}^n),
	\end{equation*}
where $J_{\mathcal{F}}$ is the Jacobian of $\mathcal{F}$ and $(-\Delta)^s$ is the fractional Laplacian operator (see, e.g., page 117 of \cite{EStein}).

	There are several works on nonlocal problems involving these conformally invariant operators of fractional order and $s$-curvature, e.g., the singular fractional Yamabe problem, the fractional Yamabe flow and the fractional Nirenberg problem are investigated (see \cite{CK, GQ, JLX1, JLX2, JX} and references therein).
	
In this work, we will consider the following problem
\begin{equation}
\label{g1.1}
 L^g_s u := P^g_s u - R^g_s u =f(\zeta,u) \text{ in } \mathbb{S}^n,
\end{equation}
where $g \in [g_{\mathbb{S}^n}]$, $s\in (0,n/2)$ and  $f:\mathbb{S}^n\times \mathbb{R} \rightarrow \mathbb{R}$ is a continuous function verifying the following conditions:
\begin{itemize}
\item[$(F_1)$] $f(\zeta,-t)=-f(\zeta, t)$ for all $\zeta\in \mathbb{S}^n$ and  $t\in \mathbb{R}$;
\item[$(F_2)$] For $p\in (1,(n+2)/(n-2s)]$, there exists a positive constant $C$ and  such that 
\begin{equation*}
|f(\zeta, t)|\leq C(1+|t|^{p}) ~\text{ for all }  \zeta\in \mathbb{S}^n \text{ and }   t\in \mathbb{R}^n;
\end{equation*}
\item[$(F_3)$] There are two constants $\mu>2$ and $R>0$ such that
\begin{equation*}
0<\mu F(\zeta, t)\leq tf(\zeta, t) ~\text{ for all } \zeta\in \mathbb{S}^n \text{ and } |t|\geq R,
\end{equation*}
where $F(\zeta,t)=\int_{0}^{t}f(\zeta,\tau)d\tau$ for all $\zeta\in \mathbb{S}^n$ and $t\in \mathbb{R}$.
\end{itemize}

If $s=1$ and $f(\cdot,t) = |t|^{p-1}t - \lambda t$, then \eqref{g1.1} becomes
\begin{equation}
\label{g1.2}
-\Delta_g u =u^{p} -\lambda u,~u>0 \text{ in } \mathbb{S}^n.
\end{equation}  
When $p=(n+2)/(n-2)$, the problem (\ref{g1.2}) is related to the so called Brezis-Nirenberg problem if $\lambda$ is the negative constant, whereas, if $4(n-1)\lambda=(n-2)R_g$, where $R_g$ denotes the scalar curvature of $\mathbb{S}^n$, then (\ref{g1.2}) is just the classical Yamabe problem in the conformal geometry. In the case where $g=g_{\mathbb{S}^n}$ is the standard metric,  there are infinitely many solutions to the Yamabe problem
with respect to the metric $g$ since the conformal group of the standard unit $n$-sphere is also infinite. On the other hand, in the spirit of Lin-Ni’s conjecture \cite{LN}, it was shown by Brezis and Li \cite{BL} for the case $n=3$ and by Hebey \cite{EH2} for $n\geq 4$ with $R_g>0$ that there exists a $\overline{\lambda}>0$ such that (\ref{g1.2}) has only a constant solution for $0<\lambda<\overline{\lambda}$. Such kind of uniqueness results for (\ref{g1.2}) was also shown by Licois and V\'{e}ron \cite{LV} for the subcritical case, that is, $1<p<(n+2)/(n-2)$. Other versions of such results can be found in \cite{BV, DEL, GS}.

 Brezis and Li \cite{BL} also studied the problem \eqref{g1.1} in the standard sphere $(\mathbb{S}^n,g_{\mathbb{S}^n})$. They showed that  \eqref{g1.1} admits only constant solutions provided that $s=1$, $f$ is such that $f(\cdot,t)=f(t)$ and the function $h(t) = t^{-(n+2)/(n-2)}(f(t) + n(n - 2)t/4)$ is decreasing on $(0,+\infty)$. Hence, considering the particular problem (\ref{g1.2}) on the standard sphere, they showed that if $0 < \lambda < n(n-2)/4$, then the only positive solution to (\ref{g1.2}) is the constant $u\equiv \lambda^{1/(p-1)}$. Recently, we generalized the Brezis-Li's results for systems in the case $s\in (0,n/2)$, see \cite{ABR} . 

 Throughout the paper, we assume that
 \begin{equation}
 \label{g1.40}
 \lambda_{1,s,g}=\min_{u\in H^s(\mathbb{S}^n,g)}\frac{\int_{\mathbb{S}^n}(u-\overline{u})(P^g_s(u-\overline{u})-R^g_s(u-\overline{u}))d\upsilon_g}{\int_{\mathbb{S}^n}(u-\overline{u})^2d\upsilon_g}>0,
 \end{equation}
where $s\in(1,n/2)$, $H^s(\mathbb{S}^n, g)$ is the completion of the space of smooth functions $C^{\infty}(\mathbb{S}^n)$ under the norm $\|\cdot\|_{s,g}$ defined by
\begin{equation}
\label{g1.10}
\|u\|^2_{s,g}:=\int_{\mathbb{S}^n}uP^g_s u~d\upsilon_g, ~u\in C^{\infty}(\mathbb{S}^n),
\end{equation}
and $\overline{u}:=\avint u~d\upsilon_g$. If $s=1$, the value $\lambda_{1,1,g}$ is the first positive eigenvalue of the operator $P^g_1-R^g_1=-\Delta_g$ on $(\mathbb{S}^n,g)$. When $s\in(0,1)$, Pavlov and Samko \cite{PS} showed that
\begin{equation}
	\label{g6.15}
	P_s(u)(\zeta)=C_{n,-s}P.V.\int_{{\mathbb{S}^n}}\frac{u(\zeta)-u(z)}{|\zeta-z|^{n+2s}}d\upsilon_{g_{\mathbb{S}^n}}^{(z)} + P_s(1)u(\zeta),~u\in C^2({\mathbb{S}^n}),~\zeta\in {\mathbb{S}^n},
	\end{equation}
	where $C_{n,-s}=\frac{2^{2s}s\Gamma(\frac{n+2s}{2})}{\pi^{\frac{n}{2}}\Gamma(1-s)}$, $|\cdot|$ is the Euclidean distance in $\mathbb{R}^{n+1}$ and $P.V.\int_{{\mathbb{S}^n}}$ is understood as $\lim_{\varepsilon \rightarrow 0}\int_{|x-y|>\varepsilon}$.
 From (\ref{g6.15})  one can without difficulty check that (\ref{g1.40}) holds for $s\in (0,1)$ (see \cite{ABR}). In the case $g=g_{\mathbb{S}^n}$, the value of $\lambda_{1,s,g}$ is determined by spherical harmonics (see Appendix).

Motivated by the above results we can state our first result about the uniqueness of solutions for \eqref{g1.1} as follows.
\begin{theorem}
\label{gt2.1}
Let $g\in [g_{\mathbb{S}^n}]$ be a conformal metric on $\mathbb{S}^n$. Assume that $R^g_s$ is positive for $p=(n+2s)/(n-2s)$. If $s\in (0,n/2)$ and $p\in(1,(n+2s)/(n-2s)]$, then there exists some $\lambda^*=\lambda^*(n,s,g, \mathbb{S}^n)>0$ such that the only positive solution of \eqref{g1.1}, with $f(\cdot,t) = t^p - \lambda t$, is the constant solution $u=\lambda^{1/(p-1)}$ for $0<\lambda<\lambda^*$ on $(\mathbb{S}^n,g)$.
\end{theorem}

If $g=g_{\mathbb{S}^n}$, it was shown in \cite{ABR} that $\lambda^* = \Gamma(n/2+s)/\Gamma(n/2-s)$. Moreover, if $p$ is a subcritical exponent, then the conclusion of the above theorem holds for $\lambda=\lambda^*$ on $(\mathbb{S}^n,g_{\mathbb{S}^n})$.

 The proof of Theorem \ref{gt2.1} consists in two cases: subcritical and critical. The subcritical case was shown in \cite{ABR}. In the critical case we will use methods of blow up analysis and compactness of solutions for any $s\in(0,n/2)$ that will be developed in Section \ref{gs4}. In both cases it is essential the positiveness on $\mathbb{S}^n$ of the solution $u$. Such condition leads to the following question: Does the conclusion of Theorem \ref{gt2.1} hold if the solution of \eqref{g1.1} is positive somewhere?
 
To answer the question it is necessary to analyze the following problem
\begin{equation}
\label{1.2}
\begin{cases}
(-\Delta)^s v=|v|^{2^*_s-2}v \text{ in } \mathbb{R}^n,\\
u\in D^{s,2}(\mathbb{R}^n),
\end{cases}
\end{equation}
where $s\in(0,n/2)$, $n>2$, $2^*_s=2n/(n-2s)$ and $D^{s,2}(\mathbb{R}^n)$ denotes the space of real-valued functions $u\in L^{2^*_s}(\mathbb{R}^n)$ whose energy associated to $(-\Delta)^s$ is finite, that is, 
\begin{equation}
\label{1.3}
\|v\|^2_s:=\int_{\mathbb{R}^n}v(-\Delta)^s v~dx =\int_{\mathbb{R}^n} |x|^{2s}|\widehat{v}|^2~dx< +\infty,
\end{equation}
where $~\widehat{ }~$ denotes the Fourier transform. The space $D^{s,2}(\mathbb{R}^n)$ can be also seen as the completion of the space of smooth function with compact support in $\mathbb{R}^n$ under the norm (\ref{1.3}). Note that, if $s=1$, then $(-\Delta)^s = -\Delta$  is the classical Laplacian operator and (\ref{1.3}) becomes
\begin{equation*}
\|v\|^2_1=\int_{\mathbb{R}^n}|\nabla v|^2~dx ,
\end{equation*}
and if $s \in (0,1)$, then 
\begin{equation*}
\|v\|^2_s:=\int_{\mathbb{R}^n}\int_{\mathbb{R}^n}\frac{[v(x)-v(y)]^2}{|x-y|^{n+2s}}dx dy.
\end{equation*}
 As we will see, the problem \eqref{g1.1} contains the problem \eqref{1.2}. In \cite{GNN}, Gidas, Ni and Nirenberg proved that any positive solution of
\begin{equation}
\label{1.4}
-\Delta v = |v|^{\frac{4}{n-2}}v ~\text{ in } \mathbb{R}^n,
\end{equation}
with some decay at infinity, and by Caffarelli, Gidas and Spruck \cite{CGS} without condition at infinity are necessarily of the form
\begin{equation}
\label{1.5}
v(x)=\frac{[n(n-2)a^2]^{\frac{n-2}{4}}}{(a^2+|x-x_0|^2)^{\frac{n-2}{2}}},
\end{equation}
where $a>0$, $x_0\in \mathbb{R}^n$. For some generalizations of this result when the right hand side of \eqref{1.4} is a function well behaved, or $\mathbb{R}^n$ is substituted by $\mathbb{R}^n_+$, see  Damascelli and Gladiali \cite{DG} and references therein. This question also was addressed in a more general setting by Lin \cite{Lin}, Wei and Xu \cite{WX}, by considering the polyharmonic operator when $n>2m$, $m\in \mathbb{N}$. They used the moving plane method to prove that all positive solutions of the problem 
\begin{equation}
\label{1.6}
(-\Delta)^m v= |v|^{\frac{4m}{n-2m}}v ~ \text{ in } \mathbb{R}^n, v\in D^{2,m}(\mathbb{R}^n),
\end{equation}
 take the form
\begin{equation*}
v(x)=\frac{C_{n,m}}{(a^2+|x-x_0|^2)^{\frac{n-2m}{2}}},
\end{equation*}
where $C_{n,m}$ is a constant depending of $n$ and $m$. Later, Guo and Liu \cite{GLp} generalized Wei and Xu's results. 

With the assumption that the solutions have finite energy,
Ding \cite{DW} showed that \eqref{1.4} has an unbounded
sequence of solutions that are different from those given by (\ref{1.5}). Following the same idea of \cite{DW}, Bartsch, Schneider and Weth \cite{BSW} proved the existence of infinitely many sign-changing weak solutions of (\ref{1.6}).

In the fractional setting, for $s\in(0,n/2)$, Chen, Li and Ou \cite{CLO1} showed that the positive solutions of \eqref{1.2} are the form
\begin{equation}\label{fractional}
u(x)=\frac{C_{n,s}}{(a^2+|x-x_0|^2)^{\frac{n-2s}{2}}}.
\end{equation}
After that, Chen, Li and  Yu \cite{CLc, XY3} generalized this result by showing the nonexistence of positive solutions for a class of nonlocal equations. 

However, the existence of sign-changing solutions of \eqref{g1.1} is not yet proven for the case where $s\notin \mathbb{N}$. Some works on the study of the existence of sign-changing solutions for some equations involving the fractional Laplace operator and for fractional Yamabe problems can be referred to in \cite{CW, GMM, Kris, Maa} and references therein. Motivate by the cited papers, we can state our main result about the existence of solutions as follows.

\begin{theorem}
\label{t1.1}
Let $g$ be the standard metric on $\mathbb{S}^n$. If $s\in (0,n/2)$, then there exists an unbounded sequence $\{u_{l}\}_{l\in \mathbb{N}}$ in $H^{s}(\mathbb{S}^n,g)$ of solutions of \eqref{g1.1} provided that $f(\cdot,t)=f(t)$. In particular, the problem \eqref{1.2} has an unbounded sequence $\{v_{l}\}_{l\in \mathbb{N}}$ in $D^{s,2}(\mathbb{R}^n)$ of sign-changing solutions for any $s\in (0,n/2)$.
\end{theorem}

We remark that the equations \eqref{1.2}, \eqref{1.4} and (\ref{1.6}) are invariant under the conformal transformations on $\mathbb{R}^n$. Moreover, for each equation, their respective positive solutions obtained have the same energy.  For $g=g_{\mathbb{S}^n}$, it was shown  in \cite{ABR} that the positive solutions to problem \eqref{g1.1} are constants provided that $f(\cdot,t)=f(t)$ and the function $\tilde{h}(t)=t^{-\frac{n+2s}{n-2s}}[f(t)+R^{g_{\mathbb{S}^n}}_s t]$ is decreasing on $(0,+\infty)$. As a counter part of this result we have:

\begin{corollary}
\label{c1.1}
Let $g$ be the standard metric on $\mathbb{S}^n$ and $s\in(0,n/2)$. Assume that $f(\cdot,t)=f(t)$ is nondecreasing for $s>1$ and the function $ h(t)=t^{-\frac{n+2s}{n-2s}}[f(t)+R^{g
}_s t]$ is decreasing on $(0,+\infty)$ for $s>0$. Then there exists an unbounded sequence $\{u_{l}\}_{l\in \mathbb{N}}$ in $H^{s}(\mathbb{S}^n, g)$ of sign-changing solutions of \eqref{g1.1}.
\end{corollary}

The function $f$ considered in Theorem \ref{gt2.1} is a typical example for this case. So, \eqref{g1.1} becomes
\begin{equation}
\label{1.7}
P_s u - R_s u=|u|^{p-1}u-\lambda u ~\text{ in } \mathbb{S}^n,
\end{equation}
where $R_s$ denotes the $s$-curvature on the standard sphere $(\mathbb{S}^n,g_{\mathbb{S}^n})$, $0<\lambda \leq \lambda^*$ and $\lambda^*$ is given by Theorem \ref{gt2.1}. Then by Corollary \ref{c1.1}, Theorem \ref{gt2.1} and its remarks, we have:
\begin{corollary}
 Let $g$ be the standard metric on $\mathbb{S}^n$ and $s\in(0,n/2)$. Assume that $p \leq (n+2)/(n - 2s)$, $\lambda \leq  \lambda^*$ and at least one of these inequalities is strict. Then there exists an unbounded sequence $\{u_{l}\}_{l\in \mathbb{N}}$ in $H^{s}(\mathbb{S}^n,g)$ of sign-changing solutions of \eqref{1.7}.
\end{corollary}

The proof of our above results consists in use some standard variational techniques combined with the symmetries of $\mathbb{S}^n$, which plays an important role because it permits us to reduce a problem on $\mathbb{R}^n$ to an equivalent problem on the standard unit sphere $\mathbb{S}^n$, and vice versa. 

Using all the arguments mentioned jointly with the fountain theorem \cite{BW}, and the principle of symmetric criticality \cite{P1}, we will generalize Theorem \ref{t1.1} for the case $g\in [g_{\mathbb{S}^n}]$. For this purpose, we introduce some notations. Given a unit sphere $(\mathbb{S}^n, g)$ which is conformally equivalent to the standard unit sphere, we denote by $Isom_g(\mathbb{S}^n)$ the group of isometries of $(\mathbb{S}^n,g)$. It is well known 
  that $Isom_g(\mathbb{S}^n)$ is a compact Lie group and its action is differentiable on $\mathbb{S}^n$. Moreover, the compact Lie group $O(n+1)$ acts isometrically on $\mathbb{S}^n$. For a closed subgroup $G$ of $Isom_g(\mathbb{S}^n)$, we set
\begin{equation*}
H^s_G(\mathbb{S}^n,g):=\{ u\in H^s(\mathbb{S}^n,g);~ \forall \sigma\in G, \sigma u = u\}.
\end{equation*}
For $x\in \mathbb{S}^n$, we denote the orbit of $x$ under the action of $G$ by
\begin{equation}
O^x_G =: \{ \sigma (x); ~\sigma\in G\},
\end{equation}
and the isotropy group of $x$ by
\begin{equation}
S^x_G =: \{ \sigma\in G;~ \sigma(x)=x\}.
\end{equation}

Now we state our following result that generalizes the Theorem \ref{t1.1}. Furthermore, such result implies results on the existence of infinite different solutions for some fractional or non-local problems on $\mathbb{R}^n$.

\begin{theorem}
\label{gt1.5}
Let $G\subset O(n+1)$ be a closed subgroup with $k=\min_{x\in \mathbb{S}^n}\dim O^x_G\geq 1$ and let $g=\varphi^{\frac{4}{n-2s}}g_{\mathbb{S}^n}$ be a conformal metric on $\mathbb{S}^n$ with $0<\varphi\in C^{\infty}(\mathbb{S}^n)$ and $s \in (0,n/2)$. If $\dim H^s_G(\mathbb{S}^n,g)=\infty$, $\varphi$ is $G$-invariant and $f$ is $G$-invariant on the $\zeta$-variable, then there exists an unbounded sequence $\{u_{l}\}_{l\in \mathbb{N}}$ in $H^{s}(\mathbb{S}^n,g)$ of solutions of \eqref{g1.1}.
\end{theorem}
 
The condition on the dimension of $H^s_G(\mathbb{S}^n,g)$ is essential in the proof of the above theorem. However, depending on the choice of the subgroup $G$ and the metric $g$, one can show such a condition (see Examples \ref{gex5.4}, \ref{gex5.5} and \ref{gex5.6}).  
 
\begin{corollary}
\label{gc1.6} With the same assumptions as in Theorem \ref{gt1.5} and Theorem \ref{gt2.1}. Then there exists some $\lambda^*=\lambda^*(n,s,g,\mathbb{S}^n)>0$ such that \eqref{g1.1}, with $f(\cdot,t)=|t|^{p-1}t-\lambda t$, has an unbounded sequence $\{u_{l}\}_{l\in \mathbb{N}}$ in $H^{s}(\mathbb{S}^n,g)$ of sing-changing solutions  for each $0<\lambda<\lambda^*$.
\end{corollary}

The proof of Theorem \ref{gt1.5} will be carried out in Section \ref{gs5} and is based on the arguments developed by Ding \cite{DW} for the case of the round sphere, and  by Hebey and Vaugon \cite{HV}.

\section{A problem on $\mathbb{R}^n$ equivalent to \eqref{g1.1}} \label{sg2}

Let $\mathcal{F}$ be the inverse of the stereographic projection with $S\in \mathbb{S}^n$ being the south pole. We recall that $\mathcal{F}$ is a conformal diffeomorphism. More precisely, if $g_{\mathbb{R}^n}$ denotes the flat Euclidean metric on $\mathbb{R}^n$, then the pullback of $g_{\mathbb{S}^n}$ to $\mathbb{R}^n$ satisfies
\begin{equation*}
\mathcal{F}^* g_{\mathbb{S}^n}=\frac{4}{(1+|\cdot|^2)^2}g_{\mathbb{R}^n}.
\end{equation*}
Moreover, the corresponding volume element is given by
\begin{equation}
\label{2.1}
d\upsilon_{g_{\mathbb{S}^n}}=\left(\frac{2}{1+|y|^2}\right)^n~dy.
\end{equation}
For a function $v:\mathbb{S}^n \rightarrow \mathbb{R}$, we may define
\begin{equation*}
\mathcal{P}v:\mathbb{R}^n\rightarrow \mathbb{R},~(\mathcal{P}v)(y):=\xi(y)  v(\mathcal{F}(y)),
\end{equation*}
where
\begin{equation*}
\xi(y)=\left( \frac{2}{1+|y|^2}\right)^{\frac{n-2s}{2}}.
\end{equation*}
From (\ref{2.1}), it is easy to see that $\mathcal{P}$ defines an isometric isomorphism between $L^{2^*_s}(\mathbb{S}^n,g_{\mathbb{S}^n})$ and $L^{2^*_s}(\mathbb{R}^n, g_{\mathbb{R}^n})$. Denote by $H^s(\mathbb{S}^n)$ the completion of the space of smooth functions $C^{\infty}(\mathbb{S}^n)$ under the norm $\| . \|_{*}$ induced by the scalar product 
\begin{equation}
\label{2.2}
\langle v,w\rangle_{*}:=\langle \mathcal{P}v,\mathcal{P}w \rangle_{s}, ~v.w\in C^{\infty}(\mathbb{S}^n),
\end{equation}
where $\langle \cdot, \cdot\rangle_s$ is defined by (\ref{1.3}). Hence, by construction, $\mathcal{P}$ is also an isometric isomorphism between $(H^s(\mathbb{S}^n), \| \cdot \|_{*})$ and $(D^{s,2}(\mathbb{R}^n), \| \cdot \|_s)$.

We note that $\langle \cdot,\cdot \rangle_{*}$ is the quadratic form of a unique positive self adjoint
operator in $L^2(\mathbb{S}^n,g_{\mathbb{S}^n})$ which coincides with $P_s$, $H^s(\mathbb{S}^n)=H^s(\mathbb{S}, g_{\mathbb{S}^n})$,
\begin{equation}
\label{2.3}
\|v\|^2_*=\int_{\mathbb{S}^n}vP_sv~d\upsilon_{g_{\mathbb{S}^n}}~\text{ for all } v\in H^s(\mathbb{S}^n)
\end{equation} 
and 
\begin{equation*}
(P_sv)\circ \mathcal{F}=\xi^{-\frac{n+2s}{n-2s}} (-\Delta)^s (\mathcal{P}v)~\text{ for all } v\in C^{\infty}(\mathbb{S}^n).
\end{equation*}
By (\ref{r1.1}) we have
\begin{equation}
\label{g2.3}
(P^g_s u)\circ \mathcal{F} = (\mathcal{P}\varphi)^{-\frac{n+2s}{n-2s}}(-\Delta)^s(\mathcal{P} (\varphi u)) \text{ for all } \varphi, u\in C^{\infty}(\mathbb{S}^n) \text{ with } g=\varphi^{\frac{4}{n-2s}}g_{\mathbb{S}^n}.
\end{equation}

 Consider the problem:
\begin{equation}
\label{2.4}
\begin{cases}
(-\Delta)^s v= [\xi (\varphi \circ \mathcal{F})]^{\frac{n+2s}{n-2s}}\left( f\left(\cdot, \frac{v}{\xi (\varphi \circ \mathcal{F})}\right) + \tilde{R}^g_s \frac{v}{\xi (\varphi \circ \mathcal{F})} \right)\text{ in } \mathbb{R}^n,\\
v\in D^{s,2}(\mathbb{R}^n),~0<s<n/2,
\end{cases}
\end{equation}
where $0<\varphi \in C^{\infty}(\mathbb{S}^n)$, $\tilde{R}^g_s= R^g_s\circ \mathcal{F}$ and $ R^g_s$ is the $s$-curvature of $\mathbb{S}^n$ with $g = \varphi^{4/(n-2s)}g_{\mathbb{S}^n}$.
The following lemma constitutes the bridge between \eqref{g1.1} and (\ref{2.4}). From now on, solution means a solution in the weak sense.
\begin{lemma}
\label{l2.1}
Let $g=\varphi^{4/(n-2s)}g_{\mathbb{S}^n}$ be a conformal metric on $\mathbb{S}^n$ where $0<\varphi\in C^{\infty}(\mathbb{S}^n)$ and $s\in (0,n/2)$. Then every solution $u\in H^s(\mathbb{S}^n,g)$ of \eqref{g1.1} corresponds to a solution $v\in D^{s,2}(\mathbb{R}^n)$ of (\ref{2.4}) and
\begin{equation}
\label{2.6}
\|v\|_{s}=\|u\|_{s,g}.
\end{equation}
\end{lemma}
\begin{proof}
Let $u\in H^s(\mathbb{S}^n,g)$ be a solution of \eqref{g1.1}. Define $v(y)=\mathcal{P}( \varphi u)(y)$, $y\in \mathbb{R}^n$. We shall prove first that $v\in D^{s,2}(\mathbb{R}^n)$. By (\ref{1.3}) and (\ref{2.6}), we have
\begin{align*}
\|v\|^2_s  = \int_{\mathbb{S}^n}uP^g_su~d\upsilon_g < +\infty.
\end{align*}
By the isometry of $\mathcal{P}$, Sobolev inequality in $\mathbb{S}^n$ (see \cite{Bec, CT}) and (\ref{r1.1}),
\begin{align*}
\int_{\mathbb{R}^n}|v|^{2^*_s}dy  = \int_{\mathbb{S}^n}|\varphi u|^{2^*_s}~d\upsilon_{g_{\mathbb{S}^n}}
	 \leq C(n,s) \int_{\mathbb{S}^n}u P_s u~d\upsilon_g<+\infty.
\end{align*}
Thus $v\in D^{s,2}(\mathbb{R}^n)$ and (\ref{2.6}) follows.

We remark that the volume element for $g$ is given by
 \begin{equation*}
d\upsilon_g = (\xi (\varphi \circ \mathcal{F}))^{\frac{2n}{n-2s}}dy.
\end{equation*}
Now, we will prove that $u$ is a solution of (\ref{2.4}). Let $w\in D^{s,2}(\mathbb{R}^n)$. Then
\begin{equation*}
 \mathcal{P}^{-1}w=\frac{w\circ \mathcal{F}^{-1}}{(\xi \circ \mathcal{F}^{-1}) \varphi} \in H^s(\mathbb{S}^n, g),
\end{equation*}
 and 
\begin{align*}
& \int_{\mathbb{R}^n} [\xi (\varphi \circ \mathcal{F})]^{\frac{n+2s}{n-2s}}\left( f\left( y, \frac{v}{\xi (\varphi \circ \mathcal{F})}\right) + R^g_s \frac{v}{\xi (\varphi \circ \mathcal{F})} \right) w~dy\\
 &  =  \int_{\mathbb{R}^n} [\xi (\varphi \circ \mathcal{F})]^{\frac{n+2s}{n-2s}}\left( f\left( y, u \circ \mathcal{F}\right) + \tilde{R}^g_s (u\circ \mathcal{F}) \right) w~dy\\
	& = \int_{\mathbb{S}^n}\varphi^{-1}(f(\zeta, u) + R^g_s u )\left(\frac{w}{\xi}\circ \mathcal{F}^{-1}\right)d\upsilon_g\\
	& = \int_{\mathbb{S}^n}(f(\zeta,u) + R^g_s u ) \mathcal{P}^{-1}w~d\upsilon_g\\
	& = \langle u, \mathcal{P}^{-1}w\rangle_{s,g}=\langle \varphi u,\varphi\mathcal{P}^{-1}w  \rangle_{s,g_{\mathbb{S}^n}}\\
	& =\langle v,w  \rangle_s,
\end{align*}
where in the last two equalities we used (\ref{r1.1}), (\ref{2.3}) and (\ref{2.2}).
\end{proof}

\section{Constant solutions for a critical equation}\label{gs4}

In this section we will study the following problem
\begin{equation}
\label{pc4.1}
L^g_s u = u^{\frac{n+2s}{n-2s}}-\lambda u, ~u> 0 \text{ in } \mathbb{S}^n, 
\end{equation} 
where $g=\varphi^{4/(n-2s)}g_{\mathbb{S}^n}$, $0<\varphi\in C^{\infty}(\mathbb{S}^n)$, $s\in (0,n/2)$ and $\lambda$ is a positive constant. 
Using arguments based on Jin, Li and Xiong \cite{JLX1, JLX} about the blow up analysis and compactness, we will establish some conditions on $\lambda$ and $R^g_s$ such that (\ref{pc4.1}) has only the constant solution $u\equiv \lambda^{(n-2s)/(4s)}$.

In the case $s=1$, the compactness of solutions  of (\ref{pc4.1}) was studied by Li and Zhu \cite{LZ}  for $n=3$ and by Hebey \cite{EH2} for $n\geq 4$. Recently, similar results on compactness were obtained by Niu, Peng and Xiong \cite{NPX} for $s\in (0,1)$ with $s<n/4$.

In order to prove Theorem \ref{gt2.1}, and in view of Lemma \ref{l2.1}, we will extend the results in \cite{NPX} for $s\in (0,n/2)$ using some arguments based on analysis of isolated blow up points of solutions for the following problem
\begin{equation*}
\label{pc4.2}
(-\Delta)^s v= v^{\frac{n+2s}{n-2s}} + a v,~ v>0 \text{ in } \mathbb{R}^n,  
\end{equation*}
where $a$ is a positive smooth function on $\mathbb{R}^n$. The above problem is equivalent (up to constants) to the nonlinear integral problem
\begin{equation}
\label{pc4.3}
v(x)=\int_{\mathbb{R}^n}\frac{v(y)^{\frac{n+2s}{n-2s}}+a(y)v(y)}{|x-y|^{n-2s}}dy, ~v>0~\text{ in } \mathbb{R}^n.
\end{equation}
Thus, our main result in this section can be stated as follows.
\begin{lemma}
\label{gl8.1}
If $s\in (0,n/2)$ and $R^g_s$ is positive in $\mathbb{S}^n$, then there exist  constants positives $\tilde{\lambda}$ and $\tilde{C}$, depending only on $n$, $s$, $\inf_{\mathbb{S}^n} R^g_s$ and an upper bound of $\|R\|_{C^2(\mathbb{S}^n)}$, such that for $0<\lambda<\tilde{\lambda}$,  any solution $u$ of \eqref{pc4.1} satisfies
\begin{equation*}
\|u\|_{C^{2s}(\mathbb{S}^n)}\leq \tilde{C}.
\end{equation*}
\end{lemma}

In the above lemma we use, for simplicity, $C^s(\Omega)$ to denote $C^{\lfloor s\rfloor, s-\lfloor s\rfloor}(\Omega)$ over a domain $\Omega \subset \mathbb{R}^n$.

\subsection{ Blow up analysis of solutions of nonlinear integral equations}\

Let $\tau_i$ be a sequence of non-negative constants satisfying $\lim_{i\rightarrow \infty} \tau_i = 0$, and set
\begin{equation*}
p_i=\frac{n+2s}{n-2s}-\tau_i.
\end{equation*} 
Suppose that $0\leq v_i \in L_{loc}^{\infty}(\mathbb{R}^n)$ satisfies the nonlinear integral equation
\begin{equation}
\label{g3.1}
v_i(x)=\int_{\mathbb{R}^n}\frac{v_i^{p_i}(y)+a_i(y)v_i(y)}{|x-y|^{n-2s}}dy~\text{ in } \Omega,
\end{equation}  
where $\Omega=B_3\subset \mathbb{R}^n$ and $a_i$ are nonnegative bound functions in $\mathbb{R}^n$. We assume that $a_i\in C^2(\mathbb{R}^n)$ and, for some positive constants $A_1$ and $A_2$,
\begin{equation}
\label{g3.2}
A_1^{-1}\leq a_i \leq A_1  \text{ in } \Omega, \text{ and } \|a_i\|_{C^{1,1}(\Omega)}\leq A_2.
\end{equation}

We say that $\{v_i\}_{i\in \mathbb{N}}$ blows up if $\|v_i\|_{L^{\infty}(\Omega)}\rightarrow \infty$ as $i\rightarrow \infty$.

\begin{definition}
\label{gd3.1}
Suppose that $\{v_i\}_{i\in \mathbb{N}}$ satisfies \eqref{g3.1}. We say that a point $\tilde{x}\in \Omega$ is an isolated blow up point of  $\{v_i\}_{i\in \mathbb{N}}$ if there exist $0<\tilde{r}<dist(\tilde{x},\partial \Omega)$, $A_3>0$ and sequence $\{x_i\}_{i\in \mathbb{N}}$ such that $x_i$ is a local maximum of $v_i$, $x_i\rightarrow \tilde{x}$, $v_i(x_i)\rightarrow \infty$ and 
\begin{equation*}
v_i(x)\leq A_3|x-x_i|^{-\frac{2s}{p_i-1}} \text{ for all } x\in B_{\tilde{r}}(\tilde{x}).
\end{equation*}
\end{definition}
 
 Let $\tilde{x}$ be an isolated blow up of $\{v_i\}_{i \in \mathbb{N}}$.  We will denote the above definition by $x_i\rightarrow \tilde{x}$.  Define
 \begin{equation}
\label{g3.3}\overline{v}_i(r):=\frac{1}{|\partial B_r(x_i)|}\int_{\partial B_r(x_i)} v_i ds, ~ r>0, 
\end{equation}   
 and 
\begin{equation*}
\overline{w}_i(r):= r^{\frac{2s}{p_i-1}}\overline{v}_i(r), ~ r>0.
\end{equation*}

\begin{definition}
\label{gd3.2}
We say that $ \tilde{x} \in \Omega$ is an isolated simple blow up point, if $\tilde{x}$ is an isolated blow up point and, for some $\rho>0$, (independent of $i$) $\overline{w}_i$ has precisely one critical point in $(0,\rho)$ for large $i$.  
\end{definition}

If $x_i\rightarrow 0$ is an isolated blow up point, then we will have the following Harnack inequality in the annulus with center at $0$ whose proof is similar to \cite{JLX}.

\begin{lemma}
\label{gl3.1}
Suppose that $v_i\in L^{\infty}_{loc}(\mathbb{R}^n)$ is a nonnegative function satisfying \eqref{g3.1}. Suppose that $x_i\rightarrow 0$ is an  isolated simple blow up point of  $\{v_i\}_{i \in \mathbb{N}}$. Then, for any $0<r<\tilde{r}/3$, we have the following Harnack inequality
\begin{equation*}
\sup_{B_{2r}(x_i)\backslash \overline{B_{r/2}(x_i)}}v_i \leq C \inf_{B_{2r}(x_i)\backslash \overline{B_{r/2}(x_i)}}v_i,
\end{equation*} 
where $C$ is a positive constant depending only on $n$, $s$, $\tilde{r}$, $A_3$ and  $A_2$.
\end{lemma} 

 The next two results are similar to \cite[Proposition 2.9 and 2.10]{JLX}. The first one shows that if $x_i\rightarrow 0 $ is an isolated blow up point, then $v_i$ tends to be a standard bubble. Then second one shows that $v_i$ is a lower bound near isolated blow up points.
\begin{proposition}
\label{gp3.1}
Assume the assumptions in Lemma \ref{gl3.1}. Then for any $R_i\rightarrow \infty$, $\varepsilon_i \rightarrow 0^+$, we have , unless subsequence, that 
\begin{equation*}
\| m_i^{-1} v_i(m_i^{-\frac{p_i- 1}{2s}} \cdot + x_i) - (1+|\cdot|^2)^{\frac{2s-n}{2}} \|_{C^2(B_{2R_i}(0))}\leq \varepsilon_i
\end{equation*}
\begin{equation*}
r_i:=R_i m_i^{-\frac{p_i-1}{2s}}\rightarrow 0 \text{ as } i \rightarrow \infty,
\end{equation*}
where $m_i=v_i(x_i)$.
\end{proposition}

\begin{proposition}
\label{gp3.2}
Under the hypotheses of Proposition \ref{gp3.1}, there exists some positive constant $C=C(n,s,A_1, A_2, A_3)$ such that
\begin{equation*}
v_i(x)\geq C m_i(1+m_i^{\frac{p_i-1}{s}}|x-x_i|^2)^{\frac{2s-n}{2}} \text{ for all } |x-x_i|\leq 1.
\end{equation*}
In particular, for any $e\in \mathbb{R}^n$, $|e|=1$, we have
\begin{equation*}
v_i(x_i+e) \geq C m_i^{-1+\frac{(n-2s)\tau_i}{2s}},
\end{equation*}
where $\tau_i=(n+2s)/(n-2s)-p_i$.
\end{proposition}
 
 To obtain an upper bound of $v_i$ near isolated blow up points, we need auxiliary bounds and a variant of a Pohoz\v{a}ev type identity.
 \begin{lemma}
 \label{gl3.6}
 Under the hypotheses of Proposition \ref{gp3.1} with $\tilde{r}=2$, and  in addition that $x_i\rightarrow 0$ is also an isolated simple blow up point with the constant $\rho$, then there exists $\delta_i>0$, $\delta_i=O(R_i^{-2s+o(1)})$, such that
 \begin{equation*}
 v_i(x)\leq C R_i^{(n-2s)\tau_i}v_i(x_i)^{-\lambda_i}|x-x_i|^{2s-n+\delta_i} \text{ for all } r_i\leq |x-x_i|\leq 1, 
 \end{equation*}
 where $\lambda_i = (n-2s-\delta_i)(p_i-1)/2s - 1$ and $C=C(n,s, A_1, A_3, \rho)>0$.
 \end{lemma}
 \begin{proof}
 The proof consists of four steps. Steps 1, 3 and 4 are similar as in \cite{JLX}.

STEP 2. From step 1, we have 
\begin{equation}
\label{g3.6}
v_i(x)^{p_i-1} \leq C R_i^{-2s+o(1)}|x-x_i|^{-2s} \text{ for all } r_i\leq |x-x_i|\leq \rho,
\end{equation}
 where $o(1)$ denotes some quantity tending to $0$ as $i\rightarrow \infty$. 
 
Now,  let 
\begin{equation*}
L_i \phi(y):=\int_{\mathbb{R}^n}\frac{v_i(z)^{p_i-1}\phi(z)+a_i(z)\phi(z)}{|y-z|^{n-2s}}dz.
\end{equation*}
Thus \begin{equation*}
v_i=L_i v_i.
\end{equation*}
Note that for $2s<\alpha<n$ and $0<|x|<2$,
\begin{align*}
\int_{\mathbb{R}^n}\frac{1}{|x-y|^{n-2s}|y|^{\alpha}} dy& = |x|^{2s-n}\int_{\mathbb{R}^n}\frac{1}{||x|^{-1}x-|x|^{-1}y|^{n-2s}|y|^{\alpha}}dy \\
& = |x|^{-\alpha+2s} \int_{\mathbb{R}^n}\frac{1}{||x|^{-1}x-z|^{n-2s}|z|^{\alpha}}dz\\
& \leq C\left(\frac{1}{n-\alpha}+\frac{1}{\alpha-2s}+1 \right)|x|^{-\alpha+2s},
\end{align*}
where we made the change of variables $y=|x|z$. By (\ref{g3.6}), we can choose $0<\delta_i=O(R_i^{-2s+o(1)})$ such that
\begin{equation}
\label{g3.7}
\int_{r_i<|y-x_i|<\rho}\frac{v_i(y)^{p_i-1}|y-x_i|^{-\delta}}{|x-y|^{n-2s}}dy \leq \frac{1}{4}|x-x_i|^{-\delta_i},
\end{equation} 
and 
\begin{equation}
\label{g3.8}
\int_{r_i<|y-x_i|<\rho}\frac{v_i(y)^{p_i-1}|y-x_i|^{2s-n+\delta}}{|x-y|^{n-2s}}dy \leq \frac{1}{4}|x-x_i|^{2s-n+\delta_i},
\end{equation} 
for all $r_i<|x-x_i|<\rho$. 

Let 
\begin{equation*}
\begin{aligned}
\Omega_{1i} & =\{ y\in \mathbb{R}^n;~r_i<|y-x_i|<\rho,~|x_i-y|>|x-x_i|/2,~|x-y|<|x-x_i|/2\},\\
\Omega_{2i} & =\{y\in \mathbb{R}^n;~r_i<|y-x_i|<\rho,~|x_i-y|<|x-x_i|/2,~|x-y|>|x-x_i|/2\},\\
\Omega_{3i} & =\{ x\in \mathbb{R}^n;~r_i<|y-x_i|<\rho,~|x_i-y|>|x-x_i|/2,~|x-y|>|x-x_i|/2\},\\
\end{aligned}
\end{equation*}
where  $r_i<|x-x_i|<\rho$. Then
\begin{equation}
\label{g3.9}
\begin{aligned}
\int_{r_i<|y-x_i|<\rho}\frac{|y-x_i|^{-\delta_i}}{|x-y|^{n-2s}}dy &\leq C \int_{\Omega_{1i}}+\int_{\Omega_{2i}}+ \int_{\Omega_{3i}}\frac{|y-x_i|^{-\delta_i}}{|x-y|^{n-2s}}dy \\
&\leq  C \rho^{2s}|x-x_i|^{-\delta_i},
\end{aligned}
\end{equation}
where $C$ is a positive constant depending only on $n$ and $s$. Similarly
\begin{equation}
\label{g3.10}
\int_{r_i<|y-x_i|<\rho}\frac{|y-x_i|^{2s-n+\delta_i}}{|x-y|^{n-2s}}dy \leq  C \rho^{2s}|x-x_i|^{2s-n+\delta_i}.
\end{equation}

Set $M_i:=4\times 2^{n-2s}\max_{\partial B_{\rho_1}(x_i)}u_i$,
\begin{equation*}
f_i(x):=\{ M_i\rho_1^{\delta_i}|x-x_i|^{-\delta_i}+AR_i^{(n-2s)\tau_i}m_i^{-\lambda_i} |x-x_i|^{2s-n+\delta_i}\},
\end{equation*}
and 
\begin{equation*}
\phi_i(x):=\begin{cases}
					f_i(x), & r_i<|x-x_i|<\rho_1;\\
					u_i(x), & \text{otherwise},
					\end{cases}
\end{equation*}
where $0<\rho_1<\rho$ small enough and $A>1$ will be chosen later.

We can replace $\rho$ by $\rho_1$ in (\ref{g3.9}) and (\ref{g3.10}). Then from (\ref{g3.7})-(\ref{g3.10}) and  for $r_i<|x-x_i|<\rho_1$, we have
\begin{equation*}
\begin{aligned}
&L_i \phi_i(x)\\
&  = \int_{|y-x_i|\leq r_i}+\int_{r_i<|y-x_i|<\rho_1}+\int_{\rho_1\leq |y-x_i|} \frac{v_i(y)^{p_i-1}\phi(y)+a_i(y)\phi(y)}{|x-y|^{n-2s}}dy \\
& \leq C \int_{|y-x_i|\leq r_i} \frac{v_i(y)^{p_i}+v_i(y)}{|x-y|^{n-2s}}dy+\frac{f_i}{4}+\int_{\rho_1\leq |y-x_i|} \frac{v_i(y)^{p_i-1}\phi(y)+a_i(y)\phi(y)}{|x-y|^{n-2s}}dy
\end{aligned}
\end{equation*}
To estimate the third term, we define $\tilde{x}=\rho_1 \frac{x}{|x|}+x_i\in \partial B_{\rho_1}(x_i)$, and then
\begin{equation}
\label{g3.11}
\begin{aligned}
& \int_{\rho_1\leq |y-x_i|} \frac{v_i(y)^{p_i-1}\phi(y)+a_i(y)\phi(y)}{|x-y|^{n-2s}}dy\\
& =  \int_{\rho_1\leq |y-x_i|}\frac{|\tilde{x}-y|^{n-2s}}{|x-y|^{n-2s}} \frac{v_i(y)^{p_i-1}\phi(y)+a_i(y)\phi(y)}{|\tilde{x}-y|^{n-2s}}dy\\
& \leq 2^{n-2s} \int_{\rho_1\leq |y-x_i|} \frac{v_i(y)^{p_i-1}\phi(y)+a_i(y)\phi(y)}{|\tilde{x}-y|^{n-2s}}dy\\
& \leq 2^{n-2s}v_i(\tilde{x})\leq 2^{n-2s}\max_{\partial B_{\rho_1}(x_i)}v_i\leq \frac{M_i}{4}.
\end{aligned}
\end{equation}
 To estimate the first term, we use change of variables, Proposition \ref{gp3.1}, $R_i^{n-2s}=O(m_i)$ and some computations so that
\begin{equation}
\begin{aligned}
&\int_{|y-x_i|\leq r_i} \frac{v_i(y)^{p_i-1}\phi(y)+a_i(y)\phi(y)}{|x-y|^{n-2s}}dy\\
& \leq C m_i\int_{|z|\leq R_i}\frac{[m_i^{-1}v_i(m_i^{-\frac{p_i-1}{2s}}z+x_i)]^{p_i}+m_i^{-p_i}v_i(m_i^{-\frac{p_i-1}{2s}}z+x_i)}{|m_i^{\frac{p_i-1}{2s}}(x-x_i)-z|^{n-2s}}dz\\
& \leq C m_i\int_{|z|\leq R_i}\frac{\xi(z)^{p_i}+m_i^{-p_i}m_i \xi(z)}{|m_i^{\frac{p_i-1}{2s}}(x-x_i)-z|^{n-2s}}dz\\
& \leq C m_i\int_{|z|\leq R_i}\frac{\xi(z)^{p_i}}{|m_i^{\frac{p_i-1}{2s}}(x-x_i)-z|^{n-2s}}dz\\
& \leq C m_i R_i^{(n-2s)\tau_i}\int_{\mathbb{R}^n}\frac{\xi(z)^{\frac{n+2s}{n-2s}}}{|m_i^{\frac{p_i-1}{2s}}(x-x_i)-z|^{n-2s}}dz\\
& = Cm_iR_i^{(n-2s)\tau_i}\xi(m_i^{\frac{p_i-1}{2s}}(x-x_i)),
\end{aligned}
\end{equation} 
 where 
 \begin{equation*}
 \xi_s(z)=\left( \frac{2}{1+|z|^2} \right)^{\frac{n-2s}{2}}.
 \end{equation*}
 Since $|x-x_i|\geq r_i$, we see that
 \begin{equation*}
\begin{aligned}
m_i \xi_s(m_i^{\frac{p_i-1}{2s}}(x-x_i)) &  \leq C m_i^{\lambda_i}|x-x_i|^{2s-n+\delta_i}.  
\end{aligned} 
 \end{equation*}
 Therefore, we conclude that 
 \begin{equation*}
 L_i \phi_i(x)\leq \phi_i(x) ~\text{ for all } r_i<|x-x_i|<\rho_1,
 \end{equation*}
 provided that $A$ is sufficiently large and independent of $i$.
 \end{proof}
 
 The following identities are essential in the blow up analysis and the compactness of solutions of nonlinear differential and integral equations. Other types of identities can be found in \cite{JLX1, NPX}.
 \begin{proposition}
 \label{gp3.7}
 (Pohozaev type identities) Let $v$ be a non-negative continuous solution in $\mathbb{R}^n$ of
   \begin{equation*}
   v(x)=\int_{B_R}\frac{H(y)v(y)^p+a(y)v(y)}{|x-y|^{n-2s}}dy + h_R(x),
   \end{equation*}
   where $1<p\leq (n+2s)/(n-2s)$ and $h_R\in C^1(B_R)$, $\nabla h_R \in L^1(B_R)$. Then  we have the following:
   
  {\bf Identity 1.} 
   \begin{equation*}
   \begin{aligned}
  & \left(  \frac{n-2s}{2} - \frac{n+2}{p+1}\right) \int_{B_R}|x|^2H(x)v(x)^{p+1}dx -\frac{1}{p+1}\int_{B_R}|x|^2x\nabla H(x)v(x)^{p+1} \\
 &~\text{ }~ -(s+1)\int_{B_R}|x|^2a(x)v(x)^2dx   - \frac{1}{2}\int_{B_R}|x|^2x\nabla a(x) v(x)^2 dx\\
 &~\text{ }~ +\frac{R^3}{p+1}\int_{\partial B_R}H(x)v(x)^{p+1}d\sigma  + \frac{R^3}{2}\int_{\partial B_R}a(x)v(x)^2d\sigma \\
&=  \int_{B_R}|x|^2[H(x)v(x)^p+a(x)v(x)]\left(  \frac{n-2s}{2}h_R(x) +x\nabla h_R(x)\right)dx\\
 &~\text{ }~  -\frac{n-2s}{4}\int_{B_R}\int_{B_R}\frac{(|x|^2-|y|^2)^2}{|x-y|^{n-2s+2}}(Hv^p+av)(x)(Hv^p+av)(y)dxdy.
   \end{aligned}
   \end{equation*}
   
  {\bf Identity 2.} 
    \begin{equation*}
   \begin{aligned}
  & \left(  \frac{n-2s}{2} - \frac{n}{p+1}\right) \int_{B_R}H(x)v(x)^{p+1}dx -\frac{1}{p+1}\int_{B_R}x\nabla H(x)v(x)^{p+1} \\
 &~\text{ }~ -s\int_{B_R}a(x)v(x)^2dx   - \frac{1}{2}\int_{B_R}x\nabla a(x) v(x)^2 dx\\
 &~\text{ }~ +\frac{R}{p+1}\int_{\partial B_R}H(x)v(x)^{p+1}d\sigma  + \frac{R}{2}\int_{\partial B_R}a(x)v(x)^2d\sigma \\
  &=\int_{B_R}[H(x)v(x)^p+a(x)v(x)]\left(  \frac{n-2s}{2}h_R(x) +x\nabla h_R(x)\right)dx.
  \end{aligned}
   \end{equation*}
  
 \end{proposition}
 \begin{proof}
 The proof of Identity 2 is similar to the proof of Identity 1. We first prove the case $s>1/2$. Note that
 \begin{equation*}
 |x|^2x(x-y)=\frac{1}{2}|x|^2|x-y|^2 +\frac{1}{4}(|x|^2-|y|^2)^2 +\frac{1}{4}(|x|^4-|y|^4) \text{ for all } x,y\in \mathbb{R}^n,
 \end{equation*}
 and 
\begin{align*}
&\frac{1}{p+1}\int_{B_R} |x|^2xH(x)\nabla v^{p+1}dx + \frac{1}{2}\int_{B_R}|x|^2x a(x)\nabla v^2dx\\
& =\int_{B_R} |x|^2x [H(x)v^{p}+ a(x)v]\nabla v~dx\\
 & =(2s-n)\int_{B_R} [H(x)v^{p}+ a(x)v] \int_{B_R}\frac{|x|^2x(x-y)[H(y)v^{p}+ a(y)v]}{|x-y|^{n-2s+2}}dydx \\
& ~\text{ }~+ \int_{B_R}|x|^2x [H(x)v^{p}+ a(x)v] \nabla h_R ~dx\\
 & =\frac{2s-n}{2}\int_{B_R} |x|^2[H(x)v^{p}+ a(x)v] \int_{B_R}\frac{H(y)v^{p}+ a(y)v}{|x-y|^{n-2s}}dydx \\
 & ~\text{ }~+\frac{2s-n}{4}\int_{B_R}\int_{B_R}\frac{(|x|^2-|y|^2)^2}{|x-y|^{n-2s+2}}(Hv^p+av)(x)(Hv^p+av)(y)dxdy\\
& ~\text{ }~+ \int_{B_R}|x|^2x [H(x)v^{p}+ a(x)v] \nabla h_R ~dx\\
 & =\frac{2s-n}{2}\int_{B_R} |x|^2[H(x)v(x)^{p}+ a(x)v(x)][v(x)-h_R(x)]dx \\
 & ~\text{ }~+\frac{2s-n}{4}\int_{B_R}\int_{B_R}\frac{(|x|^2-|y|^2)^2}{|x-y|^{n-2s+2}}(Hv^p+av)(x)(Hv^p+av)(y)dxdy\\
& ~\text{ }~+ \int_{B_R}|x|^2x [H(x)v^{p}+ a(x)v] \nabla h_R ~dx.
\end{align*}
On the other hand, by the divergence theorem,
\begin{align*}
\int_{B_R} |x|^2xH(x)\nabla v(x)^{p+1}dx & = -\int_{B_R}|x|^2[(n+2)H(x)+x\nabla H(x) ]v^{p+1}dx\\
 & ~\text{ }~ +R^3\int_{\partial B_R}H(x)v(x)^{p+1}d\sigma,
\end{align*}
and 
\begin{align*}
\int_{B_R} |x|^2x a(x)\nabla v(x)^{2}dx & = -\int_{B_R}|x|^2[(n+2)a(x)+x\nabla a(x)] v(x)^{2}dx\\
 & ~\text{ }~ +R^3\int_{\partial B_R}a(x)v^{2}d\sigma.
\end{align*}
Thus, Identity 1 follows immediately for $s>1/2$.

When $0<s\leq 1/2$, Identity 1 follows from an approximation argument as in the proof of Proposition 2.12 of \cite{JLX}.
 \end{proof}

 The following result shows the behavior of $\tau_i$ involving isolated simple blow up points.
 \begin{lemma}
 \label{gl3.8}
 Under the hypotheses in Lemma \ref{gl3.6}, we have that
 \begin{equation*}
 \tau_i=O(v_i(x_i)^{-c_1+o(1)}),
 \end{equation*}
where $c_1=\min\{1, 2/(n-2s)\}$. Thus
\begin{equation*}
v_i(x_i)^{\tau_i}=1+o(1).
\end{equation*}
 \end{lemma}
\begin{proof}
By regularity theory in \cite{JLX} and Lemma \ref{gl3.6}, we have that
\begin{equation}
\label{g3.13}
\|v_i\|_{C^2(B_{3/2}(x_i)\backslash B_{1/2}(x_i))}\leq CR_i^{(n-2s)\tau_i}m_i^{-\lambda_i}.
\end{equation}
We choose $R_i\leq m_i^{o(1)}$ for work along the prove. We can write the equation \ref{g3.1} of $v_i$ as 
\begin{equation}
\label{g3.14}
v_i(x)=\int_{B_1(x_i)}\frac{v_i(y)^{p_i}+a_i(y)v_i(y)}{|x-y|^{n-2s}}dy+ h_i(x),
\end{equation}
where
\begin{equation*}
h_i(x)=\int_{|y-x_i|\geq 1}\frac{v_i(y)^{p_i}+a_i(y)v_i(y)}{|x-y|^{n-2s}}dy.
\end{equation*}

Now we will find some estimates to $h_i$ and its derivatives. By (\ref{g3.13}) and the computation in (\ref{g3.11}), we have for any $x\in B_1(x_i)$,
\begin{equation*}
h_i(x)\leq 2^{n-2s}\max_{\partial B_1(x_i)}v_i\leq C m_i^{-1+o(1)},
\end{equation*}
 for $|x-x_i|<7/8$,
\begin{equation}
\label{g3.15}
\begin{aligned}
|\nabla h_i(x)| & \leq C\int_{|y-x_i|\geq 1} \frac{v_i(y)^{p_i}+a_i(y)v_i(y)}{|x-y|^{n-2s+1}} dy \\
& \leq C\int_{|y-x_i|\geq 1} \frac{v_i(y)^{p_i}+a_i(y)v_i(y)}{|x-y|^{n-2s}} dy \\
& \leq  C\max_{\partial B_1(x_i)}v_i \leq C m_i^{-1+o(1)}.
\end{aligned}
\end{equation}
and  $7/8 \leq |x-x_i|<1$, 
\begin{equation}
\label{g3.16}
\begin{aligned}
|\nabla h_i(x)| & \leq C\int_{|y-x_i|\geq 1,~ |y-x|\geq 1/8 } + \int_{ |y-x_i|\geq 1, ~|y-x|<1/8}  \frac{v_i(y)^{p_i}+a_i(y)v_i(y)}{|x-y|^{n-2s}} dy \\
& \leq C \max_{\partial B_1(x_i)}v_i + C m_i^{-1+o(1)}\int_{ 1-|x-x_i|<|y-x|<1/8}  \frac{1}{|x-y|^{n-2s}} dy\\
& \leq \begin{cases}
			C \frac{|1-(1-|x-x_i|)^{2s-1}|}{|2s-1|} m_i^{-1+o(1)} ~ & \text{ if } s\neq \frac{1}{2}\\
			C |\text{ln}(1-|x-x_i|)|m_i^{-1+o(1)} ~ & \text{ if } s = \frac{1}{2}.\\
			\end{cases}\\
			& \leq C(1-|x-x_i|)^{-\beta} m_i^{-1+o(1)},
\end{aligned}
\end{equation}
where $0<\beta<1$ depends on $s$. Then, applying  Proposition \ref{gp3.7} Identity 2 to (\ref{g3.14}), it leads to
\begin{equation}
\label{g3.17}
\begin{aligned}
\tau_i \int_{B_1(x_i)}&  v_i(x)^{p_i+1}dx\\
 & \leq C \int_{B_1(x_i)}|x-x_i||\nabla h_i(x)|[v_i(x)^{p_i}+v_i(x)]dx \\
& +C \int_{B_1(x_i)}|x-x_i|v_i(x)^2dx + C\int_{\partial B_1(x_i)}[v_i(x)^{p_i+1}+v_i(x)^2]d\sigma.
\end{aligned}
\end{equation}
From Proposition \ref{gp3.1} and change of variables, we have that
\begin{equation}
\label{g3.18}
\begin{aligned}
\int_{B_1(x_i)} v_i(x)^{p_i+1} dx & \geq C \int_{B_{r_i}(x_i)}\frac{m_i^{p_i+1}}{[1+|m_i^{\frac{p_i-1}{2s}}(y-x_i)|^2]^{\frac{(n-2s)(p_i+1)}{2}}}dy\\
& \geq C m_i^{\tau_i(\frac{n}{2s}-1)}\int_{B_{R_i}(0)}\frac{1}{(1+|z|^2)^{\frac{(n-2s)(p_i+1)}{2}}}dz\\
& \geq C  m_i^{\tau_i(\frac{n}{2s}-1)}\geq C,
\end{aligned}
\end{equation}
\begin{equation*}
\begin{aligned}
 \int_{B_{r_i}(x_i)}&[v_i(x)^{p_i}+v_i(x)]dx \\
 & \leq C m_i^{-\frac{n(p_i-1)}{2s}}\int_{B_{R_i}(0)} [v_i(m_i^{-\frac{p_i-1}{2s}}z+x_i)]^{p_i}+v_i(m_i^{-\frac{p_i-1}{2s}}z+x_i)]dz\\
 & \leq C  m_i^{-\frac{n(p_i-1)}{2s}} \int_{B_{R_i}(0)}[ {m_i^{p_i}}{(1+|z|^2)^{-\frac{(n-2s)p_i}{2}}}+{m_i}{(1+|z|^2)^{-\frac{n-2s}{2}}}]dz\\
 & \leq C m_i^{-1+o(1)},
\end{aligned}
\end{equation*}
and 
\begin{equation*}
\int_{B_{r_i}(x_i)}|x-x_i|v_i(x)^2dx \leq C \begin{cases} m_i^{-\frac{2}{n-2s}+o(1)} & \text{ if } n \geq 4s+1\\
 m_i^{-2+o(1)}  & \text{ if } n< 4s+1.
\end{cases}
\end{equation*}
From Lemma \ref{gl3.6}, (\ref{g3.15}) and (\ref{g3.16}),
\begin{equation*}
\begin{aligned}
\int_{r_i<|x-x_i|<1}|x-x_i|v_i(x)^2dx & \leq C m_i^{-2\lambda_i} \int_{r_i<|x-x_i|<1}|x-x_i|(|x-x_i|^{2s-n+\delta_i})^2dx\\
& \leq C \begin{cases}
				 m_i^{-\frac{2}{n-2s}+ o(1)} & \text{ if } n> 4s+1\\
				m_i^{-2\lambda_i}  & \text{ if } n\leq 4s+1,
			\end{cases}
\end{aligned}
\end{equation*}
\begin{equation*}
\begin{aligned}
\int_{r_i<|x-x_i|<1}|\nabla h_i(x)|&[v_i(x)^{p_i}+v_i(x)]dx \\
&   \leq C \int_{r_i<|x-x_i|<\frac{7}{8}} + \int_{\frac{7}{8}<|x-x_i|<1}|\nabla h_i(x)|[v_i(x)^{p_i}+v_i(x)]dx\\
& \leq C m_i^{-2\lambda_i}
\end{aligned}
\end{equation*}
and 
\begin{equation*}
\int_{\partial B_1(x_i)}[v_i(x)^{p_i+1}+v_i(x)^2]d\sigma \leq C m_i^{-2\lambda_i+o(1)}.
\end{equation*}
Combining the above estimates in (\ref{g3.17}) and the fact $\tau_i=o(1)$, the lemma follows immediately.
\end{proof} 
 
 To obtain our desired upper bound of $v_i$ we need the next result.
\begin{lemma}
\label{gl3.9}
 Under the assumptions in Lemma \ref{gl3.6} we have that
 \begin{equation*}
 \limsup_{i \rightarrow \infty} \max_{x\in \partial B_{\theta}(x_i)}v_i(x)v_i(x_i)\leq C(\theta), \text{ for each } 0<\theta\leq 1.
 \end{equation*}
\end{lemma} 
\begin{proof}
By Lemma \ref{gl3.1}, it suffices to show the lemma for sufficiently small $\theta >0$. Suppose the contrary, then along a subsequence we have that
\begin{equation*}
\lim_{i\rightarrow \infty} v_i(\theta e+x_i)v_i(x_i)=+\infty \text{ with } |e|=1.
\end{equation*}
Since $v_i(x)\leq A_3 |x-x_i|^{-2s/(p_i-1)}$ in $B_2(x_i)$, it follows from Lemma \ref{gl3.1} that for any $0<\varepsilon<\theta$, there exists a positive constant $C_{\varepsilon}=C_{\varepsilon}(n,s,A_1,A_2,A_3,\varepsilon)$ such that
\begin{equation}
\label{g3.19}
\sup_{B_{\frac{5}{2}}(x_i)\backslash B_{\varepsilon}(x_i)}v_i \leq C_{\varepsilon} \inf_{B_{\frac{5}{2}}(x_i)\backslash B_{\varepsilon}(x_i)}v_i.
\end{equation}
Moreover, from Lemma \ref{gl3.6} and Lemma \ref{gl3.8},
\begin{equation}
\label{g3.20}
\begin{aligned}
v_i(\theta e+x_i)^{p_i-1}\rightarrow 0 & \text{ as } i\rightarrow +\infty,\\  v_i(\theta e+x_i)^{\tilde{\delta_i}} \leq C & \text{ for all }  i,
\end{aligned}
\end{equation}
where  $\tilde{\delta_i}= \delta_i(p_i-1)/2s$, $\delta_i$ is as in Lemma \ref{gl3.6} and $C$ is a positive constant that does not depends on $i$.

Let $\phi_i(x)= v_i( \theta e +x_i )^{-1+\tilde{\delta_i}}v_i(x)$. Then
\begin{equation*}
\phi_i(x) = \int_{\mathbb{R}^n}\frac{v_i(\theta e +x_i)^{(p_i-1)(1-\tilde{\delta_i})}\phi_i(y)^{p_i}+a_i(y)\phi_i(y)}{|x-y|^{n-2s}}dy,
\end{equation*}
and by (\ref{g3.19}) and (\ref{g3.20})
\begin{equation}
\label{g3.21} 
\phi_i(\theta e+x_i) = O(1) \text{ and } \|\phi_i\|_{L^{\infty}(B_{\frac{3}{2}}(x_i)\backslash B_{\varepsilon}(x_i))}\leq C_{\varepsilon} \text{ for } 0<\varepsilon<\theta.
\end{equation}

By regularity theory in \cite{JLX} to $\phi_i$, there exists $\phi \in C^2(B_1\backslash \{0\})$ such that, unless of subsequence, $\phi_i\rightarrow \phi$ in $C^2_{loc}(B_1\backslash \{0\})$.
By (\ref{g3.21}) and some computations as in (\ref{g3.11}), (\ref{g3.15}) and (\ref{g3.16}), there exists $h\in C^1(B_1)$ such that, unless of subsequence, 
\begin{equation}
\label{g3.22}
h_i\rightarrow h\geq 0 \text{ in } C^1_{loc}(B_1).
\end{equation}
Therefore,
\begin{equation*}
 \int_{B_1(x_i)}\frac{v_i(\theta e +x_i)^{(p_i-1)(1-\tilde{\delta_i})}\phi_i(y)^{p_i}+a_i(y)\phi_i(y)}{|x-y|^{n-2s}}dy = \phi_i(x)-h_i(x) \rightarrow \phi(x)-h(x) 
\end{equation*}
in $C^1_{loc}(B_1\backslash \{0\})$. 

On other hand, from (\ref{g3.20}), Lemma \ref{gl3.6} and Lemma \ref{gl3.8}, 
\begin{equation}
\label{g3.23}
\begin{aligned}
\int_{B_1(x_i)\backslash B_{\varepsilon}(x_i)}\frac{a_i(y)\phi_i(y)}{|x-y|^{n-2s}} dy & \leq  \frac{C \varepsilon^{2s-n+\delta_i}}{[v_i(\theta e+x_i)v_i(x_i)]^{1-\tilde{\delta_i}}}\int_{B_1(x_i)\backslash   B_{\varepsilon}(x_i)}\frac{1}{|x-y|^{n-2s}}dy \\
& \leq   \frac{C\varepsilon^{2s-n+\delta_i}}{[v_i(\theta e+x_i)v_i(x_i)]^{1-\tilde{\delta_i}}} \rightarrow 0 \text{ as } i\rightarrow +\infty.
\end{aligned}
\end{equation}

Now we define $G(x):=\phi(x)-h(x)$. We will determine the explicit form of $G$. For any $|x|> 0$, $2|x_i|<\varepsilon<|x-x_i|/2$ and $3\varepsilon<|x|$, from (\ref{g3.20}),  (\ref{g3.21}) and (\ref{g3.23}),  we have that
\begin{equation*}
\begin{aligned}
G(x) & = \lim_{i\rightarrow \infty}\left( \int_{B_1(x_i)\backslash B_{\varepsilon}(x_i)} + \int_{B_{\varepsilon}(x_i)} \frac{v_i(\theta e +x_i)^{(p_i-1)(1-\tilde{\delta_i})}\phi_i(y)^{p_i}+a_i(y)\phi_i(y)}{|x-y|^{n-2s}}dy \right)\\
& = \lim_{i\rightarrow \infty} \int_{B_{\varepsilon}(x_i)} \frac{v_i(\theta e +x_i)^{(p_i-1)(1-\tilde{\delta_i})}\phi_i(y)^{p_i}+a_i(y)\phi_i(y)}{|x-y|^{n-2s}}dy\\
& = |x|^{2s-n}(1+O(\varepsilon))\lim_{i\rightarrow \infty} \int_{B_{\varepsilon}(x_i)} [v_i(\theta e +x_i)^{(p_i-1)(1-\tilde{\delta_i})}\phi_i(y)^{p_i}+a_i(y)\phi_i(y)]dy \\
& = |x|^{2s-n}(1+O(\varepsilon))a(\varepsilon),
\end{aligned}
\end{equation*} 
for some non-negative function $a(\varepsilon)$ of $\varepsilon$. Since $a(\varepsilon)$ is nondecreasing, so $\lim_{\varepsilon\rightarrow 0}a(\varepsilon)$ exists and we denote as $a$. Taking $\varepsilon \rightarrow 0$, we have that
\begin{equation*}
G(x)=a|x|^{2s-n}.
\end{equation*}
Since $x_i\rightarrow 0$ is an isolated simple blow point, it follows from Proposition \ref{gp3.1} that $r^{(n-2s)/2}\overline{\phi}(r)$ is nonincreasing for all $0 <r<\rho$, i.e., for any $0 <r_1 \leq r_2<\rho$,
\begin{equation*}
r_1^{(n-2s)/2}\overline{\phi}(r_1)\geq r_2^{(n-2s)/2}\overline{\phi}(r_2).
\end{equation*}
It follows that $\phi$ has to have a singularity at $0$, and thus, $a>0$. Hence,
\begin{equation*}
\lim_{i\rightarrow \infty}\int_{B_{\frac{1}{8}}(x_i)} [v_i(\theta e +x_i)^{(p_i-1)(1-\tilde{\delta_i})}\phi_i(y)^{p_i}+a_i(y)\phi_i(y)] dy > C^{-1}G\left(\frac{e}{2} \right) >0.
\end{equation*}
However, from Proposition \ref{gp3.1}, Lemma \ref{gl3.6} and Lemma \ref{gl3.8}, 
\begin{equation*}
\begin{aligned}
\int_{B_{\frac{1}{8}}(x_i)} & [v_i(\theta e +x_i)^{(p_i-1)(1-\tilde{\delta_i})}\phi_i(y)^{p_i}+a_i(y)\phi_i(y)] dy \\
& \leq C v_i(\theta e +x_i)^{-1+\tilde{\delta_i}}\int_{B_{\frac{1}{8}}(x_i)}[v_i(y)^{p_i}+v_i(y)] dy\\
& \leq C \frac{1}{[ v_i(\theta e +x_i)v_i(x_i)]^{1-\tilde{\delta_i}}}\rightarrow 0 \text{ as } i\rightarrow \infty.
\end{aligned}
\end{equation*}
This is a contradiction.

Therefore, the lemma is proved.
\end{proof}
   
 \begin{proposition}
 \label{gp3.10}
 Under the assumptions in Lemma \ref{gl3.6} we have that
 \begin{equation*}
 v_i(x)\leq C v_i(x_i)^{-1}|x-x_i|^{2s-n} \text{ for  all } |x-x_i|\leq 1.
 \end{equation*}
 \end{proposition}
\begin{proof}
For $|x-x_i|\leq r_i$, the proposition follows immediately from Proposition \ref{gp3.1} and Lemma \ref{gl3.8}. To establish the inequality in the proposition for $r_i\leq |x-x_i|\leq 1$, we will suppose the contrary and scale the problem to reduce to case of $|x-x_i|=1$. Suppose that exists a subsequence $\{\tilde{x}_i\}_{i\in \mathbb{N}}$ such that $r_i\leq |\tilde{x}_i-x_i| \leq 1$ and 
\begin{equation*}
\lim_{i\rightarrow \infty} v_i(\tilde{x}_i)v_i(x_i)|\tilde{x}_i-x_i|^{n-2s} = \infty.
\end{equation*}

Set $\mu_i:=|\tilde{x}_i-x_i|$, $\tilde{v}_i(x)=\mu_i^{2s/(p_i-1)}v_i(\mu_i x + x_i)$. Then, each $\tilde{v}_i$ satisfies
\begin{equation*}
\tilde{v}_i(x)=\int_{\mathbb{R}^n}\frac{\tilde{v}_i(y)^{p_i}+\mu_i^{2s}a_i(\mu_i y+x_i)\tilde{v}_i(y)}{|x-y|^{n-2s}}dy.
\end{equation*}
Thus, we have all hypotheses of Lemma \ref{gl3.9}  for $\{\tilde{v}_i\}_{i\in \mathbb{N}}$, with $0$ being an isolated simple blow point. Hence,  
\begin{equation*}
\tilde{v}_i(0)\tilde{v}_i(\mu_i^{-1}(\tilde{x}_i-x_i))\leq C.
\end{equation*}
It follows that
\begin{equation*}
\lim_{i\rightarrow \infty} v_i(\tilde{x}_i)v_i(x_i)|\tilde{x}_i-x_i|^{n-2s}\leq +\infty.
\end{equation*}
This is a contradiction, which concludes the proof.
\end{proof} 
 
 \begin{corollary}
 \label{gc3.11}
 Assume as in Lemma \ref{gl3.6}. We have
 \begin{equation*}
 \int_{B_{1}(x_i)}|x-x_i|^{\alpha}v_i(x)^{p_i+1}dy  = \begin{cases}
 																											O(v_i(x_i)^{-\frac{2\alpha}{n-2s}}), & \text{ if } -n<\alpha<n,\\
 																											O(v_i(x_i)^{-\frac{2n}{n-2s}}\ln v_i(x_i)), & \text{ if } \alpha =n,\\
 																											O(v_i(x_i)^{-\frac{2n}{n-2s}}), &\text{ if } \alpha>n,
 																											
 																											\end{cases}
 \end{equation*}
 and
  \begin{equation*}
 \int_{B_{1}(x_i)}|x-x_i|^{\alpha}v_i(x)^{2}dy  = \begin{cases}
 																											O(v_i(x_i)^{-\frac{4s+2\alpha}{n-2s}}), & \text{ if } 4s+\alpha<n,\\
 																											O(v_i(x_i)^{-2}\ln v_i(x_i)), & \text{ if } 4s+\alpha=n,\\
 																											O(v_i(x_i)^{-2}), &\text{ if } 4s+\alpha>n.
 																											
 																											\end{cases}
 \end{equation*}
 \end{corollary}
 \begin{proof}
 It follows from Proposition \ref{gp3.1}, Lemma \ref{gl3.8} and Proposition \ref{gp3.10}.
 \end{proof}
 
 By the proof of Lemma \ref{gl3.9}, we have the following corollary.
 \begin{corollary}
 \label{gc3.12}
 Assume as in Lemma \ref{gl3.6}. Let $T_i(x)=T_i'(x)+T_{i}''(x)$, where
 \begin{equation*}
 T_i'(x)=v_i(x_i)\int_{B_1(x_i)}\frac{v_i^{p_i}(y)+\mu_i a_i(y)v_i(y)}{|x-y|^{n-2s}}dy,
 \end{equation*}
  \begin{equation*}
 T_i''(x)=v_i(x_i)\int_{\mathbb{R}^n \backslash B_1(x_i)}\frac{v_i^{p_i}(y)+\mu_i a_i(y)v_i(y)}{|x-y|^{n-2s}}dy,
 \end{equation*}
 and $\mu_i\rightarrow 0$ as $i\rightarrow \infty$.
 Then, unless of subsequence, 
 \begin{equation*}
 T_i'(x) \rightarrow d_0|x|^{2s-n} \text{ in } C^2_{loc}(B_1\backslash \{0\})
 \end{equation*}
 and 
 \begin{equation*}
 T_i''(x)\rightarrow h(x) \text{ in } C^1_{loc}(B_1)
 \end{equation*}
 for some $h\in C^1(B_1)$, where
 \begin{equation}
 d_0=\left( \frac{\pi^\frac{n}{2}\Gamma(s)}{\Gamma(\frac{n}{2}+s)} \right)^{-\frac{n}{2s}}\int_{\mathbb{R}^n}(1+|y|^2)^{\frac{2s-n}{2}}dy.
 \end{equation}
 Consequently, we have that
\begin{equation}
v_i(x_i)v_i(x)\rightarrow d_0|x|^{2s-n}+h(x) \text{ in } C^1_{loc}(B_1\backslash \{0\}).
\end{equation}
 \end{corollary}
 \begin{proof}
 As in the proof of Lemma \ref{gl3.9}, we set $\phi_i(x)=v_i(x_i)v_i(x)$, which satisfies
 \begin{equation*}
 \begin{aligned}
 \phi_i(x) & = \int_{\mathbb{R}^n}\frac{v_i(x_i)^{1-p_i}v_i(y)^{p_i}+\mu_i a_i(y)\phi_i(y)}{|x-y|^{n-2s}}dy\\
 & =\int_{B_1(x_i)}\frac{v_i(x_i)^{1-p_i}\phi_i(y)^{p_i}+\mu_i a_i(y)\phi_i(y)}{|x-y|^{n-2s}}dy+h_i(x) = T_i'(x)+T_i''(x).
 \end{aligned}
\end{equation*}  
Then we have all the ingredients as in the proof of Lemma \ref{gl3.9}. Besides, we can use the Proposition \ref{gp3.10} to simplify the computations.  

On the other hand, the value of  the positive constant $d$ is shown in \cite[Corollary 2.19]{JLX}.
 \end{proof}

\subsection{Isolated blow up points have to be isolated simple} \
 
 Let $v_i$ be a nonnegative smooth solutions of 
 \begin{equation}
 \label{g3.26}
 v_i(x)=\int_{\mathbb{R}^n}\frac{H_i^{\tau_i}(y)v_i^{p_i}(y)+ a_i(y)v_i(y)}{|x-y|^{n-2s}}dy \text{ for } x\in \Omega=B_3,
 \end{equation}
 where $a_i$ satisfies (\ref{g3.2}), and, for some positive constants $A_4$ and $A_5$, $H_i\in C^{1,1}(\Omega)$ satisfies
 \begin{equation}
 \label{g3.27}
 A_4^{-1}\leq H_i(x) \leq A_4 \text{ for } x\in \Omega,~\|H_i\|_{C^{1,1}(\Omega)}\leq A_5.
 \end{equation}
 We assume that $x_i\rightarrow 0$ is an isolated blow up point of $\{v_i\}_{i\in \mathbb{N}}$. If $x_i\rightarrow 0$ is an isolated simple blow up point, then Proposition \ref{gp3.1} and Proposition \ref{gp3.10} hold for the solutions $v_i$ of (\ref{g3.26}).
 
 The following result shows that an isolated blow up point is also an isolated simple blow up point.
 \begin{proposition}
 \label{gp3.13}
Let $v_i$ be a nonnegative solution of \eqref{g3.26}, where $H_i$ and $a_i$ satisfy \eqref{g3.2} and \eqref{g3.27}, respectively. Suppose that $x_i\rightarrow 0$ is an isolated blow up point of $\{v_i\}_{i\in \mathbb{N}}$ in $B_2$ for some constant positive constant $A_3$, i.e., 
 \begin{equation*}
 |x-x_i|^{\frac{2s}{p_i-1}}v_i(x)\leq A_3 \text{ in } B_2.
 \end{equation*}
 Then $x_i\rightarrow 0$ is an isolated simple blow up point.
 \end{proposition}
 
  In order to prove the above proposition, we need the following estimates.
\begin{lemma}
\label{gl3.14}
  Assume as in Proposition \ref{gp3.13}. In addition, suppose that $x_i\rightarrow 0$ is an isolated simple blow up point. Then, for $i$ large enough,
  \begin{equation*}
  \tau_i \leq C m_i^{-2} +C (|\nabla a_i(x_i)|+ \|\nabla a_i\|_{C^{0,1}(B_2)})\times\begin{cases}
  			m_i^{-\frac{4s+2}{n-2s}} & \text{if } 4s+1 < n,\\
  			m_i^{-2}\ln m_i & \text{if } 4s+1=n,\\
  			m_i^{-2} & \text{if } 4s+1> n,\\
  			\end{cases}
  \end{equation*}
  where $m_i=v_i(x_i)$.
\end{lemma}
\begin{proof}
Let us write equation (\ref{g3.26}) as 
\begin{equation}
\label{g3.28}
v_i(x)=\int_{B_1(x_i)}\frac{H_i(y)^{\tau_i }v_i(y)^{p_i}+a_i(y)v_i(y)}{|x-y|^{n-2s}}dy+h_i(x),
\end{equation}
where 
\begin{equation*}
h_i(x)=\int_{\mathbb{R}^n\backslash B_1(x_i)}\frac{H_i(y)^{\tau_i }v_i(y)^{p_i}+a_i(y)v_i(y)}{|x-y|^{n-2s}}dy.
\end{equation*}
From Proposition \ref{gp3.10} and the proof of (\ref{g3.15}) and (\ref{g3.16}), we have that
\begin{equation}
\label{g3.29}
|\nabla h_i(x)| \leq C\begin{cases}
										m_i^{-1} & \text{ if } 0<|x-x_i|<\frac{7}{8},\\
										(1-|x-x_i|^{-\beta})m_i^{-1} &\text{ if } \frac{7}{8}<|x-x_i|<1,
										\end{cases}
\end{equation}
where $0<\beta<1$ depends on $s$. Using Proposition \ref{gp3.7} Identity 2, Lemma \ref{gl3.8}, Proposition \ref{gp3.10} and the same proof of Lemma \ref{gl3.8}, we have
\begin{align*}
\notag
 \tau_i &\leq C m_i^{-2}+C\int_{B_1(x_i)}|x-x_i||\nabla H_i(x)^{\tau_i}|v_i(x)^{p_i+1}dx \notag \\
 & + C\int_{B_1(x_i)}|x-x_i||\nabla a_i(x)|v_i^2dx  \notag \\
& \leq C m_i^{-2}+C\tau_i \int_{B_1(x_i)}|x-x_i|v_i(x)^{p_i+1}dx +C|\nabla a_i(x_i)| \int_{B_1(x_i)}|x-x_i|v_i(x)^2dx \notag \\
&  + C\int_{B_1(x_i)}|x-x_i||\nabla a_i(x) -\nabla a_i(x_i)|v_i(x)^2dx. 
\end{align*}
By Corollary \ref{gc3.11}, the second term of the right hand side is less than $\tau_i/2$ for $i $ large enough, and thus,
\begin{equation*}
\begin{aligned}
 \tau_i & \leq C m_i^{-2} + C|\nabla a_i(x_i)| \int_{B_1(x_i)}|x-x_i|v_i(x)^2dx \\
  &+C \|\nabla a_i\|_{C^{0,1}(B_2)} \int_{B_1(x_i)}|x-x_i|^2v_i(x)^2dx.
 \end{aligned}
\end{equation*}
Therefore, the lemma follows from Corollary \ref{gc3.11}.
\end{proof}

\begin{lemma}
  \label{gl3.15}
   Assume as in Lemma \ref{gl3.14}. Then
    \begin{align*}
| \nabla  a_i(x_i)| \leq C \begin{cases}
  						((\ln m_i)^{-1})(1+\|\nabla a_i\|_{C^{0,1}(B_2)}) & \text{ if } n=4s,\\
  						m_i^{-2+\frac{4s}{n-2s}}(1+\|\nabla a_i\|_{C^{0,1}(B_2)}) & \text{ if } 4s<n<4s+1,\\
  						m_i^{-2+\frac{4s}{n-2s}}(1+\|\nabla a_i\|_{C^{0,1}(B_2)}\ln m_i) & \text{ if } n=4s+1,\\
  						m_i^{-2+\frac{4s}{n-2s}}+m_i^{-\frac{4}{n-2s}}\|\nabla a_i\|_{C^{0,1}(B_2)} & \text{ if } n>4s+1.\\
  						\end{cases}
\end{align*} 
\end{lemma}
\begin{proof}
From (\ref{g3.28}) and integrating by parts, we have
\begin{align*}
&\frac{1}{p_i+1}\int_{\partial B_1(x_i)}(x-x_i)_j H_i(x)^{\tau_i}v_i^{p_i+1}d\sigma - \frac{1}{p_i+1}\int_{B_1(x_i)}\partial_jH_i(x)^{\tau_i}v_i^{p_i+1}dx\\
& ~\text{ }~+\frac{1}{2}\int_{\partial B_1(x_i)}(x-x_i)_j a_i(x)v_i^2d\sigma - \frac{1}{2}\int_{B_1(x_i)}\partial_ja_i(x)v_i^2dx\\
&= \frac{1}{p_i+1}\int_{B_1(x_i)}H_i(x)^{\tau_i}\partial v_i^{p_i+1}dx + \frac{1}{2}\int_{B_1(x_i)}a_i(x)\partial v_i^2dx\\
&=(2s-n)\int_{B_1(x_i)}(H_i^{\tau_i}v_i^{p_i}+a_iv_i)\int_{B_1(x_i)}\frac{(x-y)_j}{|x-y|^{n-2s+2}}(H_i^{\tau_i}v_i^{p_i}+a_i v_i)dydx\\
&~\text{ }~ +\int_{B_1(x_i)}(H_i^{\tau_i}v_i^{p_i}+a_iv_i)\partial_j h_i dx \\
&= \int_{B_1(x_i)}(H_i^{\tau_i}v_i^{p_i}+a_iv_i)\partial_j h_i dx.
\end{align*}
(If $s<1/2$), then one can use approximation arguments as in the proof of Proposition \ref{gp3.7}.)
Hence,
\begin{align*}
\left| \int_{B_1(x_i)}\partial_ja_i(x)v_i^2dx  \right| & \leq C \int_{\partial B_1(x_i)} (v_i(x)^{p_i+1}+v_i(x)^2) d\sigma + C\tau_i \int_{B_1(x_i)}v_i(x)^{p_i+1}dx\\
& + C\int_{B_1(x_i)}(v_i(x)^{p_i} + v_i(x))|\nabla h_i(x)|dx.
\end{align*}
It follows from Proposition \ref{gp3.10}, Corollary \ref{gc3.11} and (\ref{g3.29}) that
\begin{equation*}
\left| \int_{B_1(x_i)}\partial_ja_i(x)v_i^2dx  \right| \leq  C m_i^{-2} + C \tau_i.
\end{equation*}
Using the triangle inequality, we have that
\begin{align}
\notag |\nabla a_i(x_i)| & \int_{B_1(x_i)}v_i^2dx\\
 & \leq C \left| \int_{B_1(x_i)}\nabla a_i(x)v_i^2dx \right| + C\|\nabla a_i\|_{C^{0,1}(B_2)}\int_{B_1(x_i)} |x-x_i|v_i^2dx \\
\label{g3.30}
& \leq C m_i^{-2} + C\tau_i + C\|\nabla a_i\|_{C^{0,1}(B_2)}\int_{B_1(x_i)} |x-x_i|v_i(x)^2dx.
\end{align}
By Proposition \ref{gp3.2}, Lemma \ref{gl3.8} and change of variables,
\begin{equation}
\label{g3.31}
\begin{aligned}
\int_{B_1(x_i)}v_i(x)^2dx & \geq C^{-1} m_i^{-\frac{4s}{n-2s}}\int_{B_{m_i^{\frac{p_i-1}{2s}}}}(1+|y|^2)^{2s-n}dy\\
& \geq C^{-1} \begin{cases}
m_i^{-2}\ln m_i & \text{ if } 4s=n\\
m_i^{-\frac{4s}{n-2s}} & \text{ if } 4s<n.
\end{cases}
\end{aligned}
\end{equation}

Therefore, for $i$ large enough, the lemma follows from Corollary \ref{gc3.11}  and combining Lemma \ref{gl3.14}, (\ref{g3.30}) and (\ref{g3.31}). 
\end{proof}

\begin{lemma}
\label{gl3.16}
  Assume as in Lemma \ref{gl3.14}. Then
  \begin{align*}
  a_i(x_i)\leq C \begin{cases}
  						((\ln m_i)^{-1})(1+\|\nabla a_i\|_{C^{0,1}(B_2)}) & \text{ if } n=4s,\\
  						m_i^{-2+\frac{4s}{n-2s}}(1+\|\nabla a_i\|_{C^{0,1}(B_2)}) & \text{ if } 4s<n<4s+1,\\
  						m_i^{-2+\frac{4s}{n-2s}}(1+\|\nabla a_i\|_{C^{0,1}(B_2)}\ln m_i) & \text{ if } n=4s+1,\\
  						m_i^{-2+\frac{4s}{n-2s}}+m_i^{-\frac{4}{n-2s}}\|\nabla a_i\|_{C^{0,1}(B_2)} & \text{ if } n>4s+1.\\
  						\end{cases}
\end{align*}   
\end{lemma}
\begin{proof}
By Proposition \ref{gp3.7} Identity 2 and Corollary \ref{gc3.11}, we have that
\begin{equation}
\begin{aligned}
\label{g3.32}
\int_{B_1(x_i)}a_i(x)v_i^2 & \leq C m_i^{-2} + \tau_i  + C \int_{B_1(x_i)}|x-x_i||\nabla a_i(x)|v_i^2(x)dx\\
& \leq C m_i^{-2} +C\tau_i + (|\nabla a_i(x_i)|+\|\nabla a_i\|_{C^{0,1}(B_2)}) \int_{B_1(x_i)}|x-x_i|v_i^2dx.
\end{aligned}
\end{equation}
Moreover,
\begin{equation}
\label{g3.33}
a_i(x_i)\int_{B_1(x_i)}v_i(x)^2dx \leq C \int_{B_1(x_i)}a_i(x)v_i^2dx + C \|\nabla a_i\|_{C^{0,1}(B_2)} \int_{B_1(x_i)}|x-x_i|v_i^2dx.
\end{equation}
The lemma follows from (\ref{g3.33}), (\ref{g3.32}), (\ref{g3.31}), Lemma \ref{gl3.14} and Lemma \ref{gl3.15}.
\end{proof}

 {\it Proof of Proposition \ref{gp3.13}.} 
 By Proposition \ref{gp3.1}, $r^{2s/(p_i-1)}\overline{v}_i(r)$ has precisely one critical point in the interval $0 <r<r_i:=R_i v_i(x_i)^{-(p_i-1)/(2s)}$. If the proposition were wrong, let $\mu_i $ be the second critical point of $r^{2s/(p_i-1)}\overline{v}_i(r)$. Then we see that
 \begin{equation}
 \mu_i \geq r_i \text{ and } \lim_{i\rightarrow \infty} \mu_i = 0.
 \end{equation}

Without loss of generality, we assume that $x_i=0$. Set
\begin{equation*}
\phi_i(x)=\mu_i^{\frac{2s}{p_i-1}}v_i(\mu_i x), ~x\in \mathbb{R}^n.
\end{equation*} 
 Clearly, $\phi_i$ satisfies
 \begin{align*}
 \phi_i(x) & = \int_{\mathbb{R}^n}\frac{H_i^{\tau_i}(\mu_i y)\phi_i(y)^{p_i}+\mu_i^{2s}a_i(\mu_i y) \phi_i(y)}{|x-y|^{n-2s}}dy \text{ in } B_{3\mu_i^{-1}},\\
 |x|^{\frac{2s}{p_i-1}}\phi_i(x) & \leq A_3, \text{ for } |x|<2\mu_i^{-1}\rightarrow \infty,\\
 \lim_{i\rightarrow \infty} \phi_i(0)& = \infty,\\  
 r^{\frac{2s}{p_i-1}}\overline{\phi}_i(r)& \text{ has precisely one critical point in } 0<r<1,
 \end{align*}
 and
 \begin{equation*}
 \frac{d}{dr}\left\{ \left. r^{\frac{2s}{p_i-1}}\overline{\phi}_i(r) \right \}  \right|_{r=1} = 0, 
 \end{equation*}
 where $\overline{\phi}_i(r)$ is the average of $\phi_i$ over the sphere of radius $r$ with center at $0$. Therefore, $0$ is an isolated simple blow up point of $\{\phi_i\}_{i\in \mathbb{N}}$. 
 
 From Corollary \ref{gc3.12} and the proof of \cite[Proposition 2.21]{JLX}, we have that
 \begin{align*}
 \phi_i(0)\int_{B_1}\frac{H_i^{\tau_i}(\mu_i y)\phi_i(y)^{p_i} + \mu_i^{2s}a_i(\mu_i y)\phi_i(y)}{|x-y|^{n-2s}}dy  \rightarrow \frac{d_0}{|x|^{n-2s}} \text{ in } C^2_{loc}(B_1\backslash \{0\}),\\
 h_i(x)= \phi_i(0)\int_{\mathbb{R}^n\backslash B_1}\frac{H_i^{\tau_i}(\mu_i y)\phi_i(y)^{p_i} + \mu_i^{2s}a_i(\mu_i y)\phi_i(y)}{|x-y|^{n-2s}}dy  \rightarrow d_0>0 \text{ in } C^1_{loc}(B_1),
 \end{align*}
 and thus, 
 \begin{equation*}
 \phi_i(0)\phi_i(x) \rightarrow \frac{d_0}{|x|^{n-2s}}+a \text{ in } C^1_{loc}(B_1 \backslash \{0\}).
 \end{equation*}
 
In order to prove the proposition,  we are going to derive a contradiction to Pohozaev Identity 2 in Proposition \ref{gp3.7}. We need the following:\\
{ \bf Affirmation:} 
\begin{equation}
 \label{g3.35}
\lim_{i \rightarrow \infty} \phi_i(0)^2\int_{B_{\delta}}|x\nabla H_i^{\tau_i}(\mu_i x) \phi_i(x)^{p_i+1}|dx =0,
\end{equation} 
 and 
 \begin{equation}
 \label{g3.36}
\lim_{i \rightarrow \infty} \phi_i(0)^2\int_{B_{\delta}}\mu_i^{2s}|x\nabla a_i(\mu_i x) \phi_i(x)|dx =0.
\end{equation} 
{\it Proof of Affirmation.}
 Applying Lemma \ref{gl3.16} for $\phi_i$ and from the positive lower bound of $a_i$, we have that
  \begin{align*}
 \mu_i^{2s}\leq C \begin{cases}
 						(\ln \phi_i(0))^{-1}(1+\mu_i^{2s}\|\nabla a_i\|_{C^{0,1}(B_2)}) & \text{ if } n=4s,\\
  						\phi_i(0)^{-2+\frac{4s}{n-2s}}(1+\mu_i^{2s}\|\nabla a_i\|_{C^{0,1}(B_2)}) & \text{ if } 4s<n<4s+1,\\
  						\phi_i(0)^{-2+\frac{4s}{n-2s}}(1+\mu_i^{2s}\|\nabla a_i\|_{C^{0,1}(B_2)}\ln \phi_i(0)) & \text{ if } n=4s+1,\\
  						\phi_i(0)^{-2+\frac{4s}{n-2s}}+\phi_i(0)^{-\frac{4}{n-2s}}\mu_i^{2s}\|\nabla a_i\|_{C^{0,1}(B_2)} & \text{ if } n>4s+1.\\
  						\end{cases}
\end{align*}   
Since $\lim_{i\rightarrow \infty}\mu_i= 0$, one can see that the second term of the right hand side is less than $\mu_i^{2s}/2$ for $i$ large, and then
\begin{equation*}
\mu_i^{2s}\leq C\begin{cases}
									\mu_i^{2s} &\text{ if }  4s>n,\\
									(\ln \phi_i(0))^{-1} & \text{ if } 4s = n,\\
									\phi_i(0)^{-2+\frac{4s}{n-2s}} & \text{ if } 4s<n.
								\end{cases}
\end{equation*}
So, (\ref{g3.36}) follows applying Corollary \ref{gc3.11} in the following  inequality
\begin{align*}
 \phi_i(0)^2\int_{B_{\delta}}\mu_i^{2s}|x\nabla a_i(\mu_i x) \phi_i(x)|dx & \leq C\mu_i^{2s}\phi_i(0)^2 \int_{B_1}|x|\phi_i(x)^2dx.
\end{align*}

To prove (\ref{g3.35}), we apply Lemma \ref{gl3.14} for $\phi_i$, and thus we obtain
\begin{align*}
\tau_i & \leq C \phi_i(0)^{-2} + C\mu_i^{2s}\|\nabla a_i\|_{C^{0,1}(B_2)}\times\begin{cases}
  			\phi_i(0)^{-2} & \text{if } 4s+1> n,\\
  			\phi_i(0)^{-2}\ln \phi_i(0) & \text{if } 4s+1=n,\\
  			\phi_i(0)^{-\frac{4s+2}{n-2s}} & \text{if } 4s+1 < n,\\
  			\end{cases}\\
  			& \leq C  \phi_i(0)^{-2}.
\end{align*}
So, (\ref{g3.35}) follows from the above inequality and applying Corollary \ref{gc3.11} in the following inequality
\begin{equation*}
\phi_i(0)^2\int_{B_{\delta}}|x\nabla H_i^{\tau_i}(\mu_i x) \phi_i(x)^{p_i+1}|dx \leq \tau_i\phi_i(0)^2\int_{B_{\delta}}|x| \phi_i(x)^{p_i+1}dx.
\end{equation*}

This proves the affirmation. 
\hfill $\square$
 
 Now, we continue with the proof of the proposition.  Since
 \begin{equation*}
 -\frac{\tau_i(n-2s)}{2(p_i+1)}=\frac{n-2s}{2}-\frac{n}{p_i+1}\leq 0,
 \end{equation*}
 by Affirmation we have
 \begin{equation}
 \label{gie1}
 \begin{aligned}
\lim_{i \rightarrow \infty}\phi_i(0)^2 & \left(- \frac{\tau_i(n-2s)}{2(p_i+1)}\int_{B_{\delta}} H_i(\mu_i x)^{\tau_i} \phi_i^{p_i+1}dx -s\int_{B_{\delta}} \mu_i^{2s}a_i(\mu_i x) \phi_i^2dx \right. \\
 &\left. -\frac{1}{p_i+1}\int_{B_{\delta}}x\nabla H_i^{\tau_i}(\mu_i x) \phi_i^{p_i+1}dx - \frac{1}{2}\int_{B_{\delta}}\mu_i^{2s}x\nabla a_i(\mu_i x) \phi_idx \right) \leq 0.
\end{aligned}
 \end{equation}
 
 On the other hand, if we let
 \begin{equation*}
 b_i(x):=\int_{\mathbb{R}^n\backslash B_{\delta}}\frac{H_i^{\tau_i}(\mu_i y)\phi_i(y)^{p_i} + \mu_i^{2s}a_i(\mu_i y) \phi_i(y)}{|x-y|^{n-2s}}dy,
 \end{equation*}
 then, 
 \begin{align*}
 \phi_i(0)b_i(x) & \geq  \phi_i(0)\int_{\mathbb{R}^n\backslash B_1}\frac{H_i^{\tau_i}(\mu_i y)\phi_i(y)^{p_i} + \mu_i^{2s}a_i(\mu_i y)\phi_i(y)}{|x-y|^{n-2s}}dy\\
 							& \geq \phi_i(0)\int_{\mathbb{R}^n\backslash B_1}\frac{H_i^{\tau_i}(\mu_i y)\phi_i(y)^{p_i} + \mu_i^{2s}a_i(\mu_i y)\phi_i(y)}{|y|^{n-2s}}\frac{|y|^{n-2s}}{|x-y|^{n-2s}}dy\\
 							& \geq b_i(0)\frac{1}{(1+|x|)^{n-2s}}\\
 							& \geq b_i(0) 2^{2s-n} \rightarrow d_0 2^{2s-n},
 \end{align*}
for $x\in B_{\delta}$ provided that $\delta$ is small, and
\begin{equation*}
|\nabla b_i(x)|\leq \begin{cases}
								C_{\delta}\phi_i(0)^{-1} & \text{ if } |x|\leq\frac{7\delta}{8},\\	
								C_{\delta}(\delta-|x|)^{-\beta} \phi_i(0)^{-1} & \text{ if } \frac{7\delta}{8}<|x|<\delta,
								\end{cases}
\end{equation*}
 where $0<\beta<1$ and $C_{\delta}$ a positive constant depending on $\delta$.

 Hence, from Proposition \ref{gp3.1}, Lemma \ref{gl3.8} and Proposition \ref{gp3.10}, we have
\begin{equation*}
\begin{aligned}
\int_{B_{\delta}}[H_i(\mu_i x)^{\tau_i}\phi_i^{p_i} + \mu_i^{2s}a_i(\mu_i x)\phi_i] b_i(x)dx & \geq C^{-1} d_0 \phi_i(0)^{-1}\int_{B_{\delta}}\phi_i^{p_i}dx\\
& \geq C^{-1}d_0\phi_i(0)^{-2}\int_{B_1}(1+|z|^2)^{\frac{(2s-n)p_i}{2}}dz,
\end{aligned}
\end{equation*}
\begin{equation*}
\begin{aligned}
\left| \int_{B_{\delta}}x\nabla b_i H_i(\mu_i x)^{\tau_i} \phi_i^{p_i}dx\right| & \leq C\left(\int_{B_{7\delta/8}} + \int_{B_{\delta}\backslash B_{7\delta/8}}|x||\nabla b_i|\phi_i^{p_i} dx \right) \\
& \leq C o(1)\phi_i(0)^{-2},
\end{aligned}
\end{equation*}
\begin{equation*}
\int_{\partial B_{\delta}} H_i(\mu_i x)^{\tau_i}\phi_i(x)^{p_i+1}d\sigma \leq C \phi_i(0)^{-\frac{2n}{n-2s}},
\end{equation*}
\begin{equation*}
\begin{aligned}
\left| \int_{B_{\delta}}x\nabla b_i \mu_i^{2s} a(\mu_i x) \phi_i dx\right| & \leq C\mu_i^{2s}\left(\int_{B_{7\delta/8}} + \int_{B_{\delta}\backslash B_{7\delta/8}}|x||\nabla b_i|\phi_i dx \right) \\
& \leq C \mu_i^{2s}\phi_i(0)^{-2},
\end{aligned}
\end{equation*}
and
\begin{equation*}
\int_{\partial B_{\delta}} \mu_i^{2s}a_i(\mu_i x)\phi_i(x)^2d\sigma \leq C \mu_i^{2s}\phi_i(0)^{-2}.
\end{equation*}
 It follows from the above inequalities that
 \begin{align*}
 \lim_{i\rightarrow \infty}
 & \phi_i(0)^2 \left( \frac{n-2s}{2} \int_{B_{\delta}}[H_i(\mu_i)^{\tau_i}\phi_i^{p_i}+\mu_i^{2s}a_i(\mu_i x)\phi_i]b_idx \right.\\
&  + \int_{B_{\delta}}x[H_i(\mu_i)^{\tau_i}\phi_i^{p_i}+\mu_i^{2s}a_i(\mu_i x)\phi_i]\nabla b_idx \\
& \left .-\frac{\delta}{p_i+1}\int_{B_{\delta}} H_i(\mu_i x)^{\tau_i}\phi_i^{p_i+1}d\sigma  -\frac{\delta}{2}\int_{B_{\delta}} \mu_i^{2s}a_i(\mu_i x)\phi_i^2d\sigma \right)>0.
 \end{align*}
This  contradicts (\ref{gie1}) and Pohozaev Identity 2 in Proposition \ref{gp3.7}.
  \hfill $\square$
 
\subsection{Proof of Lemma \ref{gl8.1} and Theorem \ref{gt2.1}} \

Let $u\in C^2(\mathbb{S}^n)$ be a solution of
\begin{equation}
\label{gp8.1}
P_s u = u^{p} + k u, ~u>0 \text{ in } \mathbb{S}^n,
\end{equation}
where $k: \mathbb{S}^n \rightarrow \mathbb{R}$ is a positive smooth function.

\begin{proposition}
\label{gp3.17}
Assume as above. Then for any $0<\varepsilon<1$ and $R>1$, there exist large positive constants $C_1$, $C_2$ depending on $n$, $s$, $\|k\|_{C^2(\mathbb{S}^n)}$, $\varepsilon$ and $R$ such that, if 
\begin{equation*}
\max_{\mathbb{S}^n} u\geq C_1,
\end{equation*}
then $(n+2s)/(n-2s)-p<\varepsilon$, and there exists a finite set $\mathcal{T}(v)\subset \mathbb{S}^n$ such that
\begin{itemize}
\item[(i)] If $P\in \mathcal{T}(u)$, then it is a local maximum of $u$ and in the stereographic projection coordinate system $\{ y_1,...,y_n\}$ with $P$ as the south pole,
\begin{equation}
\label{t2.e1}
\|u^{-1}(P)u(u^{-\frac{p-1}{2s}}(P) y)- (1+|y|^2)^{\frac{2s-n}{2}}\|_{C^2(B_{2R})}<\varepsilon.
\end{equation}
\item[(ii)] If $P_1$, $P_2$ belonging to $\mathcal{T}(u)$, then 
\begin{equation*}
B_{Ru(P_1)^{-\frac{p-1}{2s}}}(P_1)\cap B_{Rv(P_2)^{-\frac{p-1}{2s}}}(P_2)=\emptyset.
\end{equation*} 
\item[(iii)] $u(P)\leq C_2\{dist(P,\mathcal{T}(u))^{-\frac{2s}{p-1}}\}$ for all $P\in \mathbb{S}^n$.
\end{itemize}
\end{proposition}

The proof is standard by now, which follows from the blow-up argument as the proof of Proposition \ref{gp3.1} and Liouville theorem in Jin, Li and Xiong \cite{JLX1}. We omit it here.

\begin{proposition}
\label{gp3.18}
For any  $0<\varepsilon<1$, $R>1$ and any solution $u\in C^2(\mathbb{S}^n)$ of \eqref{gp8.1} with $\max_{\mathbb{S}^n}u>C_1$, we have
\begin{equation*}
dist(P_1,P_2)\geq d_0, \text{ for any }P_1,~P_2 \in \mathcal{P}(u) \text{ and } P_1\neq P_2, 
\end{equation*}
where $d_0$ depends only on $n$, $s$, $\varepsilon$, $R$, $\inf_{\mathbb{S}^n} k$ and an upper bound of $\|k\|_{C^{2}(\mathbb{S}^n)}$.
\end{proposition}
\begin{proof}
See the proofs of \cite[Proposition 3.2]{JLX} and Proposition \ref{gp3.13}.
\end{proof} 

 {\it Proof of Lemma \ref{gl8.1}.} We first prove that $\|u\|_{L^{\infty}(\mathbb{S}^n)}<C$. Suppose the contrary, i.e., there exist sequences $\{\lambda_i\}_{i\in \mathbb{N}}$ and $\{u_i\}_{i\in\mathbb{N}}$ such that $\lambda_i \rightarrow 0^+$, $u_i$ is solution of (\ref{pc4.1}) with $\lambda=\lambda_i$ and $\|u_i\|_{L^{\infty}(\mathbb{S}^n)}\rightarrow +\infty$ as $i \rightarrow +\infty$. Denote $\tilde{u}=\varphi u$. Then, by (\ref{r1.1}), and after passing to a subsequence, $\tilde{u}_i$ solves (\ref{gp8.1}) with $k_i = R^g_s - \lambda_i>0$ for all $i\in \mathbb{N}$. Moreover,  there exists a sequence $\{\zeta_i\}_{i\in\mathbb{N}}$ in $\mathbb{S}^n$ such that, unless of subsequence,
 \begin{equation*}
\lim_{i \rightarrow \infty} \zeta_i =\zeta_0 \text{  for some } \zeta_0 \in \mathbb{S}^n, 
 \end{equation*}
 and 
 \begin{equation*}
 \lim_{i\rightarrow \infty} \tilde{u}_i(\zeta_i)= \lim_{i\rightarrow \infty}\|\tilde{u}_i\|_{L^{\infty}(\mathbb{S}^n)} = \infty.
 \end{equation*}
 
 For any fixed $\varepsilon >0$ sufficiently small and $R>>1$, by Proposition \ref{gp3.18} we have that $Card[ \mathcal{T}(\tilde{u}_i)]$ is uniformly bounded ($Card$ denotes the cardinality). By Proposition \ref{gp3.17} (iii), we can take a sequence of elements $P_i\in \mathcal{T}(\tilde{u}_i)$ such that, after passing a subsequence, $\zeta_i = P_i$ or $\lim_{i\rightarrow \infty} P_i=\zeta_0$, and thus, $\lim_{i \rightarrow \infty}\tilde{u}_i(P_i)=\infty$. Using the stereographic projection $\mathcal{F}^{-1}$ with $P_i$ being the south pole, the equation (\ref{gp8.1}) is transformed  into 
 \begin{equation}
 \label{g3.40}
 v_i(y)=\int_{\mathbb{R}^n}\frac{\xi_s(y)^{\tau_i}v_i(y)^{p_i} + a_i(y)v_i(y)}{|x-y|^{n-2s}}dy \text{ for all } x\in \mathbb{R}^n,
 \end{equation}
 where
 \begin{equation*}
 v_i(y)=\xi_s(y)\tilde{u}_i(\mathcal{F}(y)) \text{ and } a_i(y)= \xi_s^{\frac{4s}{n-2s}}k_i(\mathcal{F}(y)), ~y\in\mathbb{R}^n.
 \end{equation*}
 It follows that $0$ is an isolated blow up point of $\{v_i\}_{i\in\mathbb{N}}$, and from Proposition \ref{gp3.13}, $0$ an isolated simple blow up point.

As in the proof of Proposition \ref{gp3.13}, we will derive a contradiction to the Pohozaev identities in Proposition \ref{gp3.7}. We first consider the case $s\leq n/4$. By Lemma \ref{gl3.14} and Corollary \ref{gc3.11} we have
\begin{equation*}
\begin{aligned}
 &\left|\int_{B_{\delta}}x\nabla \xi_s^{\tau_i} v_i^{p_i+1}dx + \int_{B_{\delta}}x\nabla a_i v_i^2dx \right| \\
&  \leq C  \begin{cases}
v_i(0)^{-2} & \text{ if } 4s\leq n< 4s+1\\
v_i(0)^{-2}+ v_i(0)^{-2}\ln v_i(0) & \text{ if } n=4s+1\\
v_i(0)^{-2}+ v_i(0)^{-\frac{4s+2}{n-2s}} & \text{ if } n>4s+1,
\end{cases}
\end{aligned}
\end{equation*}
where $0<\delta<1$. It follows from (\ref{g3.31}) and the above inequality that 
\begin{align}
\notag
\lim_{i\rightarrow \infty}& v_i(0)^2  \left( \frac{\tau_i(n-2s)}{2(p_i+1)}\int_{B_{\delta}}\xi_s^{\tau_i}v_i^{p_i+1}dx + s\int_{B_\delta}a_i(x)v_i^2dx \right.\\
\notag
& \left. \text{ }~+\frac{1}{2}\int_{B_{\delta}}x\nabla a_i v_i^2dx +\frac{1}{p_i+1}x\nabla \xi_s^{\tau_i} v_i^{p_i}dx \right) \\
\label{g3.39}
 & \geq \lim_{i\rightarrow \infty}\left( C v_i(0)^2 \int_{B_\delta}v_i^2dx - v_i(0)^2\left|\int_{B_{\delta}}x\nabla \xi_s^{\tau_i} v_i^{p_i+1}dx + \int_{B_{\delta}}x\nabla a_i v_i^2dx\right|\right) \\
 \notag
 &=\infty.
\end{align}
On the other hand, if we let
\begin{equation}
\label{ge3.85}
h_i(x)= \int_{\mathbb{R}^n\backslash B_{\delta}}\frac{\xi_s(y)^{\tau_i}v_i(y)^{p_i}+a_i(y)v_i(y)}{|x-y|^{n-2s}}dy, ~x\in \mathbb{R}^n,
\end{equation}
then
\begin{equation*}
\nabla h_i(x)\leq \begin{cases}
								C_{\delta}v_i(0)^{-1} & \text{ if } |x|\leq\frac{7\delta}{8},\\	
								C_{\delta}(\delta-|x|)^{-\beta} v_i(0)^{-1} & \text{ if } \frac{7\delta}{8}<|x|<\delta.
								\end{cases}
\end{equation*}
By Proposition \ref{gp3.1}, Lemma \ref{gl3.8} and Proposition \ref{gp3.10}, we have
\begin{align*}
v_i(0)^2\left(-\frac{n-2s}{2}\int_{B_{\delta}}(\xi_s^{\tau_i}v_i^{p_i}+a_iv_i)h_idx -\int_{B_{\delta}}x(\xi_s^{\tau_i}v_i^{p_i}+a_iv_i)\nabla h_i~ dx \right. &\\
\left.+\frac{\delta}{p_i+1}\int_{\partial B_{\delta}}\xi_s^{\tau_i}v_i^{p_i+1}d\sigma + \frac{\delta}{2}\int_{\partial B_{\delta}}a_iv_i^2d\sigma \right)& \leq C_{\delta}.
\end{align*}
This contradicts (\ref{g3.39}) and Pohozaev Identity 2 in Proposition \ref{gp3.7}.

Now we use Pohozaev Identity 1 in Proposition \ref{gp3.7} for the case $n<4s$. Note that  by Proposition \ref{gp3.2}, Lemma \ref{gl3.8} and change of variables
\begin{equation*}
\label{ge3.80}
\begin{aligned}
\int_{B_{\delta}(x_i)}|x|^2v_i(x)^{p_i+1}dx & \geq C^{-1} m_i^{p_i+1 -\frac{n(p_i-1)}{2s}}\int_{B_{m_i^{\frac{p_i-1}{2s}}}}|y|^2(1+|y|^2)^{2s-n}dy\\
& \geq C^{-1},
\end{aligned}
\end{equation*}
for some positive constant $C$.
Then
\begin{equation}
\label{ge3.81}
\begin{aligned}
\lim_{i\rightarrow \infty} & v_i(0)^2\left[ \left( \frac{n+2}{p_i+1} -\frac{n-2s}{2} \right) \int_{B_{\delta}}|x|^2\xi_s^{\tau_i}v_i^{p_i+1}dx +(s+1)\int_{B_{\delta}}|x|^2a_i v_i^2dx\right]\\
&  \geq C \lim_{i\rightarrow \infty}v_i(0)^2\int_{B_{\delta}(x_i)}|x|^2v_i^{p_i+1}dx =\infty.
\end{aligned}
\end{equation}
On the other hand, by Proposition \ref{gp3.1}, Lemma \ref{gl3.8}, Proposition \ref{gp3.10} and Corollary \ref{gc3.11}, we have
\begin{equation}
\label{ge3.82}
\begin{aligned}
v_i(0)^2 &\left(- \frac{1}{p_i+1}\int_{B_{\delta}}|x|^2x\nabla \xi_s^{\tau_i} v_i^{p_i+1}dx -\frac{1}{2}\int_{B_{\delta}}|x|^2x \nabla a_i v_i^2 dx \right.\\
& ~\text{ }-\frac{n-2s}{2}\int_{B_{\delta}}|x|^2(\xi_s^{\tau_i}v_i^{p_i}+a_iv_i)h_idx -\int_{B_{\delta}}|x|^2 x(\xi_s^{\tau_i}v_i^{p_i}+a_iv_i)\nabla h_i~ dx\\
& ~\text{ }\left. +\frac{\delta^3}{p_i+1}\int_{\partial B_{\delta}}\xi_s^{\tau_i}v_i^{p_i+1}d\sigma + \frac{\delta^3}{2}\int_{\partial B_{\delta}}a_iv_i^2d\sigma \right) \leq C
\end{aligned}
\end{equation}
and 
\begin{equation}
\label{ge3.83}
\begin{aligned}
& \frac{n-2s}{4}v_i(0)^2\int_{B_{\delta}}\int_{B_{\delta}}\frac{(|x|^2-|y|^2)^2}{|x-y|^{n-2s+2}}(\xi_s^{\tau_i}v^{p_i}+a_iv_i)(x)(\xi_s^{\tau_i}v^{p_i}+a_iv_i)(y)dxdy\\
& \leq Cv_i(0)^2 \int_{B_{\delta}}\int_{B_{\delta}}\frac{(|x|^2+|y|^2)}{|x-y|^{n-2s}}(\xi_s^{\tau_i}v^{p_i}+a_iv_i)(x)(\xi_s^{\tau_i}v^{p_i}+a_iv_i)(y)dxdy\\
& \leq  Cv_i(0)^2\int_{B_{\delta}}|x|^2(\xi_s^{\tau_i}v^{p_i}+a_iv_i)(x) \int_{B_{\delta}}\frac{\xi_s^{\tau_i}v^{p_i}+a_iv_i)(y)}{|x-y|^{n-2s}}dydx\\
& \leq  Cv_i(0)^2\int_{B_{\delta}}|x|^2(\xi_s^{\tau_i}v^{p_i}+a_iv_i)(x)(v_i(x)-h_i(x))dx \leq C,
\end{aligned}
\end{equation}
where $h_i$ is defined by (\ref{ge3.85}). So, (\ref{ge3.82}) and (\ref{ge3.83}) contradict (\ref{ge3.81}) and Pohozaev Identity 1 in Proposition \ref{gp3.7}.

Therefore, $\|u_i\|_{L^{\infty}(\mathbb{S}^n)}<C$.

Now let $\zeta$ be an arbitrary point on $\mathbb{S}^n$. Using the  stereographic projection $\mathcal{F}^{-1}$ with $\zeta$ being the south pole, $v(y)=\xi_s(y)\varphi(\mathcal{F}(y)) u\mathcal{F}(y)$, $y\in \mathbb{R}^n$, satisfies (\ref{pc4.3}). The theorem then follows from interior estimates of solutions of linear equations in \cite{ABR, JLX} to $v$ and the compactness of $\mathbb{S}^n$.
\hfill $\square$

 {\it Proof of Theorem \ref{gt2.1} (critical).}
 By (\ref{r1.1}) and Lemma \ref{gl8.1}, we have that there exist some positive constants $\tilde{C}$ and $\tilde{\lambda}$ such that $\tilde{\lambda}<\min_{\mathbb{S}^n} R^g_s$ and for $0<\lambda<\tilde{\lambda}$, any solution $u$ of (\ref{pc4.1}), satisfies
 \begin{equation}
 \label{g3.41}
 \|u\|_{C^2(\mathbb{S}^n)}\leq \tilde{C}.
 \end{equation}
 Integrating equation (\ref{pc4.1}) on $\mathbb{S}^n$ leads to, using H\"{o}lder inequality,
 \begin{equation}
  \label{g3.42}
 \|u\|_{L^{p}(\mathbb{S}^n)}\leq C\lambda^{\frac{1}{p-1}},
 \end{equation}
 where $p=(n+2s)/(n-2s)$ and $C$ depends on $n$, $s$ and $\mathbb{S}^n$.
 As in the last part of the proof of Lemma \ref{gl8.1}, we use the stereographic projection $\mathcal{F}^{-1}$, the compactness of $\mathbb{S}^n$, \cite[Theorem 2.5]{JLX}, (\ref{g3.41}) and (\ref{g3.42}) to obtain
 \begin{equation}
 \label{g3.43}
 \|u\|_{L^{\infty}(\mathbb{S}^n)}\leq C\lambda^{\frac{1}{p-1}},
 \end{equation}
 where $C$ is a positive constant that depends on $n$, $s$, $\tilde{C}$ and $\mathbb{S}^n$. 
 The theorem follows from (\ref{g3.43}) and \cite[Lemma 4.1]{ABR} for $\lambda$ small.
\hfill $\square$

\section{Existence of infinitely many non-constant solutions}\label{gs5}

The solutions of the problem \eqref{g1.1} are related with the critical point of the functional  $J: H^s(\mathbb{S},g) \rightarrow \mathbb{R}$ defined by 
\begin{equation}
\label{ge4.1}
J(u)=\frac{1}{2}\int_{\mathbb{S}^n}u L^g_su d\upsilon_g - \int_{\mathbb{S}^n}F(x,u)d\upsilon_g, ~ u\in H^s(\mathbb{S},g),
\end{equation}
where $F(\zeta,t):=\int_{0}^{t}f(\zeta,s)ds$ for each $\zeta\in \mathbb{S}^n$ and any $t\in\mathbb{R}$.  By conditions on $f$, the functional $J$ is well defined and differentiable in $H^s(\mathbb{S}^n,g)$, but it fails to satisfy the Palais-Smale compactness condition in $H^s(\mathbb{S}^n,g)$. However, applying the fountain theorem \cite{BW}, and the principle of symmetric criticality \cite{P1} we obtain the following.
\begin{lemma}
\label{gl4.1}
 Let $G$ be a group acting isometrically on $H^s(\mathbb{S}^n, g)$ such that
 \begin{enumerate}
 \item[(i)] $J$ is $G$-invariant;
\item[(ii)] the embedding $X_G\hookrightarrow L^p(\mathbb{S},g)$ is compact;
\item[(iii)] $X_G$ has infinite dimension.
 \end{enumerate}
Then $J$ has a unbounded sequence of critical points $\{v_{l}\}_{l\in \mathbb{N}}$ in $H^s(\mathbb{S}^n, g)$.
\end{lemma}

To prove the condition (i) of the above lemma, it is sufficient to show that the subgroup $G\subset O(n+1)$ acts isometrically on $H^s(\mathbb{S}^n, g)$ by the action 
\begin{equation*}
gu(\zeta):=u(g^{-1}(\zeta)), ~\zeta\in \mathbb{S}^n, ~g\in G.
\end{equation*}
\begin{lemma}
\label{gl5.10}
Let $G$ be a closed subgroup of $O(n+1)$ and $s\in (0,n/2)$. Let $g=\varphi^{4/(n-2s)}g_{\mathbb{S}^n}$ be a conformal metric on $\mathbb{S}^n$ such $\varphi$ is $G$-invariant. Then $G$ acts isometrically on $H^s(\mathbb{S}^n, g)$.
\end{lemma}
\begin{proof}
Using (\ref{g6.15}), it is easy to show that if $\sigma\in(0,1)$, then
\begin{equation}
\label{ge5.63}
P_{\sigma}(\varphi g u)(g\zeta) = P_{\sigma}(\varphi u)(\zeta) \text{ for all } \zeta \in \mathbb{S} \text{ and } g\in G.
\end{equation} 
From (\ref{r1.1}), \cite{MMG1} and (\ref{ge5.63}), we have
\begin{equation*}
\begin{aligned}
\|gu\|_{s,g} & =\int_{\mathbb{S}^n}\varphi(\zeta) (gu)(\zeta)P_s(\varphi g u)(\zeta)d\upsilon_{g_{\mathbb{S}^n}} \\
& = \int_{\mathbb{S}^n}\varphi(g\zeta) u(\zeta)P_s(\varphi g u)(g\zeta)d\upsilon_{g_{\mathbb{S}^n}} \\
& =  \int_{\mathbb{S}^n}\varphi(g\zeta) u(\zeta) \prod^{\lfloor s \rfloor}_{j=1}\left( P_1 + c_{j,s} \right)P_{s-\lfloor s \rfloor}(\varphi g u)(g\zeta)d\upsilon_{g_{\mathbb{S}^n}} \\
& = \int_{\mathbb{S}^n}(\varphi u)(\zeta)P_s(\varphi u)(\zeta)d\upsilon_{g_{\mathbb{S}^n}} ,
\end{aligned}
\end{equation*}
where $c_{j,s}$ are constants depending on $n$ and $s$.
\end{proof}

In the case $s=1$, $g=g_{\mathbb{S}^n}$ and $f(\cdot,t) = t^{(n+2)/(n-2)}-R^g_s t$,  Ding \cite{DW} considered an infinite dimensional closed subset $X_G\subset H^1(\mathbb{S}^n,g_{\mathbb{S}^n})$ and showed that the embedding $X_G\hookrightarrow L^p$ is compact. To show this compactness he used Sobolev's inequality for functions in $H^1$ on limited domains of $\mathbb{R}^k$, $k<n$, and some integral inequalities. In order to show the item (ii) of Lemma \ref{gl4.1} we will use the Sobolev's inequality in fractional spaces on domains of $\mathbb{R}^n$ (see \cite{DNPV}).

For $s\in(0,n/2)$, and $n, k$ two integers, we define $p^*=p^*(n,k,s)$ by
\begin{equation*}
p^*=\frac{(n-k)2}{n-k-2s} \text{ if } n-k>2s \text{ and } p^*=+\infty \text{ if } n-k<2s. 
\end{equation*}
When $k\geq 1$ and $2s<n$, one then has that $p^*>\frac{2n}{n-2s}$ (the critical Sobolev exponent
for the embedding of $H^s(\mathbb{S}^n, g)$ in $L^p(\mathbb{S}^n,g)$). The purpose of the following result is to evidence a subset $X_G\subset H^s(\mathbb{S}^n,g)$ such that it can satisfy the conditions of Lemma \ref{gl4.1}:

\begin{lemma}
\label{gl4.3}
 Let $g=\varphi^{\frac{4}{n-2s}}g_{\mathbb{S}^n}$ be a conformal metric on $\mathbb{S}^n$ with $0<\varphi\in C^{\infty}(\mathbb{S}^n)$ and $s\in(0,n/2)$. Let $G$ be a closed subgroup of $Isom_g(\mathbb{S}^n)$, $k=\min_{x\in \mathbb{S}^n}\dim O^x_G$ and $p^*(n,k,s)$ be as above.
\begin{enumerate}
\item[(a)] If $0<s<1$, then the embedding  $H^s_G(\mathbb{S}^n,g) \hookrightarrow L^p(\mathbb{S}^n,g)$ is continuous if $1\leq p \leq p^*(n,k,s)$ and is compact if $p < p^*(n,k,s)$.
\item[(b)] If $1\leq s< n/2$, then the embedding  $H^s_G(\mathbb{S}^n,g) \hookrightarrow L^p(\mathbb{S}^n,g)$ is continuous if $1\leq p \leq p^*(n,k,\lfloor s \rfloor p_0)p_0$ and is compact if $p < p^*(n,k, \lfloor s \rfloor p_0)p_0$, where 
\begin{equation*}
p_0=\frac{n}{n-2(s - \lfloor s \rfloor)}.
\end{equation*}
\end{enumerate} 
 
\end{lemma}
\begin{proof}
\
(a) For $0<s<1$, the proof is divided into two cases.

{\it Case 1: $g=g_{\mathbb{S}^n}$.}  If $k = 0$ the result is a straightforward consequence of the standard Sobolev embedding theorem. Hence, we assume in the sequel that $k \geq 1$. By \cite[ Lemma 1]{HV}, one then has that for each $z \in \mathbb{S}^n$ there are a chart $(U, h)$ and $\delta>0$ such that
\begin{enumerate}
\item[(i)] $h(U)= B^{r}_{6\delta}\times  B^{n-r}_{\delta}$, where $B^{r}_{6\delta}$ and $B^{n-r}_{\delta}$ are open balls of $\mathbb{R}^{r}$ and $\mathbb{R}^{n-r}$ with center at $h(z)$, and $r\in \mathbb{N}$ satisfies $r\geq k$,
\item[(ii)] $h$ and $h^{-1}$ are Lipschitz on $U$ and $ B^{r}_{6\delta}\times B^{n-r}_{\delta}$,
\item[(iii)] for any $\tilde{z}\in U$, $B^{r}_{6\delta}\times \Pi_2(h(\tilde{z}))\subset h(O^z_G\cap U)$, where $\Pi_2: \mathbb{R}^r\times \mathbb{R}^{n-r} \rightarrow \mathbb{R}^{n-r}$ is the second projection,
\item[(iv)] there is a constant $c>0$ such that $c^{-1}\delta_i^j \leq g_{ij}\leq c\delta_i^j $ as linear form, where $g_{ij}$ are the components of $g$ on $(U,h)$.  
\end{enumerate}

Now, we consider a function $G$-invariant $u \in C^{\infty}(\mathbb{S}^n)$. According to (iii), one has that for any $x, x' \in  B^{r}_{6\delta}$, and any $ y\in B^{n-r}_{\delta} $, $u\circ h^{-1}(x,y)=u\circ h^{-1}(x',y)$. Then we can define a function $\tilde{u}\in C^{\infty}(B^{n-r}_{\delta},\mathbb{R})$ such that for any $x\in B^{r}_{6\delta}$ and any $y\in  B^{n-r}_{\delta}  $,
\begin{equation*}
\tilde{u}(y)=u\circ h^{-1}(x,y).
\end{equation*}
 We then get that
\begin{align}
& \int_{U}\int_{U}\frac{[u(\zeta)-u(\omega)]^2}{|\zeta-\omega|^{n+2s}}d\upsilon_g(\zeta) d\upsilon_g(\omega) \notag \\
& =\int_{(B^r_{6\delta}\times B^{n-r}_{\delta})^2}\frac{[u\circ h^{-1}(x,y)-u\circ h^{-1}(x',y')]^2}{|h^{-1}(x,y)-h^{-1}(x',y')|^{n+2s}}\sqrt{det({g_{ij}}_{\zeta})det({g_{ij}}_{\omega})}dxdydx'dy' \notag \\
\label{3.8}
& \geq C\int_{B^{n-r}_{\delta}}\int_{B^{n-r}_{\delta}}\int_{B^{r}_{6\delta}}\int_{B^{r}_{6\delta}}\frac{[\tilde{u}(y)-\tilde{u}(y')]^2}{|(x,y)-(x',y')|^{n+2s}}dxdx'dydy'.
\end{align}
Making change variable we have 
\begin{align}
 \int_{B^{r}_{6\delta}}\int_{B^{r}_{6\delta}}&\frac{1}{|(x,y)-(x',y')|^{n+2s}} dx'dx \notag \\
& = \int_{B^{r}_{6\delta}(0)}\int_{B^{r}_{6\delta}(0)}\frac{1}{[|x-x'|^2+|y-y'|^2]^{\frac{n+2s}{2}}}dx'dx \notag \\
& = \int_{B^{r}_{6\delta}(0)}\int_{B^{r}_{6\delta}(x)}\frac{1}{[|w|^2+|y-y'|^2]^{\frac{n+2s}{2}}} dw dx \notag\\
& \geq \int_{B^{r}_{3\delta}(0)\backslash B^{r}_{2\delta}(0)}\int_{B^{r}_{6\delta}(x)}\frac{1}{[|w|^2+|y-y'|^2]^{\frac{n+2s}{2}}} dwdx \notag \\
& \geq \int_{B^{r}_{3\delta}(0)\backslash B^{r}_{2\delta}(0)}\int_{B^{r}_{\delta}(0)}\frac{1}{[|w|^2+|y-y'|^2]^{\frac{n+2s}{2}}} dwdx \notag \\
& \geq \int_{B^{r}_{3\delta}(0)\backslash B^{r}_{2\delta}(0)}\int_{B^{r}_{\frac{|y-y'|}{2}}(0)}\frac{1}{[|w|^2+|y-y'|^2]^{\frac{n+2s}{2}}} dwdx \notag \\
& \geq C \int_{B^{r}_{3\delta}(0)\backslash B^{r}_{2\delta}(0)}\frac{1}{|y-y'|^{n+2s}}\int_{B^{r}_{\frac{|y-y'|}{2}}(0)}dwdx \notag \\
& \geq C \int_{B^{r}_{3\delta}(0)\backslash B^{r}_{2\delta}(0)}\frac{1}{|y-y'|^{n-r+2s}}dx \notag \\
\label{3.9}
& \geq C\frac{1}{|y-y'|^{n-r+2s}},
\end{align}
where $C=C(\delta,n,r)>0$. From (\ref{3.8}), (\ref{3.9}) we obtain
\begin{equation*}
\int_{U}\int_{U}\frac{[u(\zeta)-u(\omega)]^2}{|\zeta-\omega|^{n+2s}}d\upsilon_g(\zeta) d\upsilon_g(\omega)\geq C_1 \int_{B^{n-r}_{\delta}}\int_{B^{n-r}_{\delta}}\frac{[\tilde{u}(y)-\tilde{u}(y')]^2}{|y-y'|^{n-r+2s}}dy'dy.
\end{equation*}
Similarly, we can prove that for any real number $p\geq 1$,
\begin{equation*}
\label{3.11}
C_2^{-1}\int_{B^{n-r}_{\delta}}|u|^pdy \leq \int_{U}|u|^p d\upsilon_g \leq C_2 \int_{B^{n-r}_{\delta}}|u|^pdy,
\end{equation*}
for some  $C_2=C_2(\delta,n,r)>0$. Combining these inequalities and the Sobolev embedding theorem for bounded domains of Euclidean spaces (see for instance \cite{DNPV}), we have that:
\begin{enumerate}
\item[(v)] if $n-r \leq 2s$, then for any real number $p\geq 1$, there exists $C>0$ such that for any $G$-invariant $u\in C^{\infty}(\mathbb{S})$,
\begin{equation*}
\left( \int_{U}|u|^p \right)^{\frac{2}{p}}\leq C \left\{\kappa_1 \int_{U}\int_{U}\frac{[u(\zeta)-u(\omega)]^2}{|\zeta-\omega|^{n+2s}}d\upsilon_g(\zeta) d\upsilon_g(\omega)+ \kappa_2\int_{U}|u|^2 d\upsilon_g \right\},
\end{equation*} 
where
\begin{equation*}
\kappa_1=\frac{	C_{n,-s}}{2}=\frac{2^{2s-1}s\Gamma(\frac{n+2s}{2})}{ \pi^{\frac{n}{2}}\Gamma(1-s)}~ \text{ and } ~\kappa_2 =\frac{\Gamma(\frac{n}{2}+s)}{\Gamma(\frac{n}{2}-s)};
\end{equation*}
\item[(vi)] if $n-r>2s$, then the inequality above hold for $1\leq p \leq \frac{(n-r)2}{n-r-2s}$.
\end{enumerate}
But we have $n-r\leq n-k$, so that $p^*(n,r,s)\geq p^*(n,k,s)$.

Finally, $\mathbb{S}^n$ can be covered by a finite number of charts $(U_{\alpha},h_{\alpha})$, $1\leq \alpha \leq N$ such that satisfy (i)-(vi). Thus, the global inequality now follows easily:
\begin{equation*}
\|u\|_{L^p(\mathbb{S}^n)}\leq C \|u\|_{s,g_{\mathbb{S}^n}}, ~1\leq p\leq p^*,
\end{equation*}
for all $u\in H^s_G(\mathbb{S},g_{\mathbb{S}^n})$. This proves that $H^s_G(\mathbb{S},g_{\mathbb{S}^n}) \hookrightarrow L^p(\mathbb{S}^n)$ is continuous for $1\leq p\leq p^*$. By standard arguments, one easily gets that the embedding is compact provided that $p < p^*$.

{\it Case 2: $g=\varphi^{\frac{4}{n-2s}}g_{\mathbb{S}^n}$.} It is easy to show that there is a $\varepsilon\in(0,1)$ such that
\begin{equation*}
\frac{(\varepsilon^2-1)}{\varepsilon^2}\kappa_1\int_{\mathbb{S}^n}\frac{[\varphi(\zeta)-\varphi(\omega)]^2}{|\zeta-\omega|^{n+2s}}d\upsilon_g(\zeta)+\kappa_2\varphi(\omega)\geq 0~\text{ for all } \omega\in \mathbb{S}^n.
\end{equation*}
By the conformal transformation relation we have that for $u\in H^s_G(\mathbb{S}^n,g)$,
\begin{align*}
&\int_{\mathbb{S}^n}uP^g_su d\upsilon_g\\
 & = \int_{\mathbb{S}^n} u\varphi P_s (u\varphi) d\upsilon_{g_{\mathbb{S}^n}}\\
&= \kappa_1\int_{\mathbb{S}^n}\int_{\mathbb{S}^n}\frac{[u(\zeta)\varphi(\zeta)-u(\omega)\varphi(\omega)]^2}{|\zeta-\omega|^{n+2s}}d\upsilon_{g_{\mathbb{S}^n}}^{(\zeta)} d\upsilon_{g_{\mathbb{S}^n}}^{(\omega)}+\kappa_2\int_{\mathbb{S}^n}u^2\varphi^2 d\upsilon_{g_{\mathbb{S}^n}}\\
& = \kappa_1\int_{\mathbb{S}^n}\int_{\mathbb{S}^n}\left\{\frac{\varphi(\zeta)^2[u(\zeta)-u(\omega)]^2}{|\zeta-\omega|^{n+2s}} + \frac{u(\omega)^2[\varphi(\zeta)-\varphi(\omega)]^2}{|\zeta-\omega|^{n+2s}} \right\} d\upsilon_{g_{\mathbb{S}^n}}^{(\zeta)} d\upsilon_{g_{\mathbb{S}^n}}^{(\omega)}\\
& + \kappa_1 \int_{\mathbb{S}^n}\int_{\mathbb{S}^n}\frac{\varphi(\zeta)[u(\zeta)-u(\omega)] u(\omega)[\varphi(\zeta)-\varphi(\omega)]}{|\zeta-\omega|^{n+2s}}  d\upsilon_{g_{\mathbb{S}^n}}^{(\zeta)} d\upsilon_{g_{\mathbb{S}^n}}^{(\omega)}+\kappa_2\int_{\mathbb{S}^n}u^2\varphi^2 d\upsilon_{g_{\mathbb{S}^n}}\\
& \geq \kappa_1\int_{\mathbb{S}^n}\int_{\mathbb{S}^n}(1-\varepsilon^2)\frac{\varphi(\zeta)^2[u(\zeta)-u(\omega)]^2}{|\zeta-\omega|^{n+2s}} d\upsilon_{g_{\mathbb{S}^n}}^{(\zeta)} d\upsilon_{g_{\mathbb{S}^n}}^{(\omega)} \\
& + \kappa_1\int_{\mathbb{S}^n}\int_{\mathbb{S}^n}\left( 1-\frac{1}{\varepsilon^2} \right) \frac{u(\omega)^2[\varphi(\zeta)-\varphi(\omega)]^2}{|\zeta-\omega|^{n+2s}}  d\upsilon_{g_{\mathbb{S}^n}}^{(\zeta)} d\upsilon_{g_{\mathbb{S}^n}}^{(\omega)} + \kappa_2\int_{\mathbb{S}^n}u^2\varphi^2 d\upsilon_{g_{\mathbb{S}^n}}\\
& \geq \kappa_1\int_{\mathbb{S}^n}\int_{\mathbb{S}^n}(1-\varepsilon^2)\frac{\varphi(\zeta)^2[u(\zeta)-u(\omega)]^2}{|\zeta-\omega|^{n+2s}} d\upsilon_{g_{\mathbb{S}^n}}^{(\zeta)} d\upsilon_{g_{\mathbb{S}^n}}^{(\omega)}\\
& + \int_{\mathbb{S}^n}u(\omega)\int_{\mathbb{S}^n}\left\{\kappa_1\left( 1-\frac{1}{\varepsilon^2} \right) \frac{[\varphi(\zeta)-\varphi(\omega)]^2}{|\zeta-\omega|^{n+2s}}+\kappa_2\varphi(\omega)\right\}d\upsilon_{g_{\mathbb{S}^n}}^{(\zeta)} d\upsilon_{g_{\mathbb{S}^n}}^{(\omega)}\\
& \geq C\int_{\mathbb{S}^n}\int_{\mathbb{S}^n}\frac{[u(\zeta)-u(\omega)]^2}{|\zeta-\omega|^{n+2s}} d\upsilon_{g}^{(\zeta)} d\upsilon_{g}^{(\omega)},
\end{align*}
where the last step is due to the fact that $\min_{\zeta\in \mathbb{S}^n}\varphi(\zeta)>0$, and $C>0$. Using the last inequality and following the arguments of the previous case, we completed the proof of item (a).

(b)  If $s=1$, this lemma is a special case of a result of Hebey and Vaugon \cite[Corollary 1]{HV}. For $s>1$, we use an iterative argument developed by Aubin \cite[Proposition 2.11]{Aubin} to prove that
\begin{equation*}
W^{\lfloor s \rfloor, q}_G(\mathbb{S},g) \hookrightarrow W^{\lfloor s \rfloor -1, q_1}_G(\mathbb{S},g) \hookrightarrow W^{\lfloor s \rfloor -2, q_2}_G(\mathbb{S},g) \hookrightarrow ... \hookrightarrow L^{q_{\lfloor s \rfloor}}(\mathbb{S},g),
\end{equation*}
where $q_i=p^*(n,k,q_{i-1}/2)q_{i-1}/2$ for $i\geq 2$ and $q_1=p^*(n,k,q/2)q/2$. An easy calculation yields $q_{\lfloor s \rfloor} = p^*(n,k,\lfloor s\rfloor q/2)q/2$. On the other hand, from theorem \ref{gat1.1} we have
\begin{equation*}
H^s_G(\mathbb{S}^n, g) \hookrightarrow H^{\lfloor s \rfloor, q}_G(\mathbb{S}^n,g)= W^{\lfloor s \rfloor, q}_G(\mathbb{S},g),
\end{equation*}
where $$q= \frac{2n}{n-2(s - \lfloor s \rfloor)}.$$
Combining these facts, we complete the proof.
\end{proof}

We mention some examples of subgroups $G\subset O(n+1)$ that act isometrically on $\mathbb{S}^n$.
\begin{example}
\label{gex5.4}
For each pair of integers $l,m\geq 2$ with $l+m=n+1$, let $G=O(l)\times O(m)\subset O(n)$. This subgroup was used by Ding \cite{DW} to establish the existence of infinitely many sign-changing solutions for the equation \eqref{1.2}. Since $k=\min_{\zeta\in \mathbb{S}^n}\dim O^{\zeta}_G\geq \min\{l-1,m-1\}$, we can see that Lemma \ref{gl4.3} holds for $g=g_{\mathbb{S}^n}$, and therefore, $H^s_G(\mathbb{S}^n, g_{\mathbb{S}^n}) \hookrightarrow L^p(\mathbb{S},g_{\mathbb{S}^n})$ is compact for $1\leq p \leq 2n/(n-2s)$. Thus, we can see that the proof of Theorem \ref{t1.1} follows immediately from Lemma \ref{gl4.1} with $X_G=H^s_G(\mathbb{S}^n, g_{\mathbb{S}^n})$.
\end{example}

\begin{example}
\label{gex5.5}
For $n=3$ or $n\geq 5$, we choose an integer $m$ between $2$ and $(n+1)/2$ with $2m \neq n$. We write the elements of $\mathbb{R}^{n+1}=\mathbb{R}^m\times \mathbb{R}^m\times \mathbb{R}^{n+1-2m}$ as $\zeta=(\zeta_1,\zeta_2,\zeta_3)$ with $\zeta_1,\zeta_2\in \mathbb{R}^m$ and $\zeta_3\in \mathbb{R}^{n+1-2m}$.  We define 
\begin{equation*}
Z=O(m)\times O(m)\times O(n+1-2m) \subset O(n+1)
\end{equation*}
and
\begin{equation*}
G=\langle Z \cup \{\tau\} \rangle \subset O(n+1),
\end{equation*}
where $\tau\in O(n+1)$ is given by $\tau(\zeta_1,\zeta_2,\zeta_3)=(\zeta_2,\zeta_1,\zeta_3)$. The subgroup $G$ was used by Bartsch and Willem \cite{BW} to find an unbounded sequence of nonradial solutions of the  equation
\begin{equation}
\label{ge5.61}
-\Delta v + b(|x|)v = \tilde{f}(|x|,v) \text{ in } \mathbb{R}^{N},~N =4 \text{ or } N\geq  6
\end{equation}
under suitable assumptions on $b$ and $\tilde{f}$. The elements of $G$ can be written uniquely as $z$ or $z\tau$  with $z\in Z$. The action on $G$ on $H^s(\mathbb{S},g_{\mathbb{S}^n})$ is defined as
\begin{equation}
\label{ge5.60}
(gu)(\zeta)=\Pi (g)u(g^{-1}\zeta) \text{ for } g\in G, ~\zeta\in \mathbb{S}^n,
\end{equation}
where $\Pi:G\rightarrow \{-1,1\}$ is given by $\Pi(z)=1$ and $\Pi(z\tau)=-1$. It is clear that each $g\in G$ acts isometrically on $H^s(\mathbb{S}^n,g_{\mathbb{S}^n})$ (see Lemma \ref{gl5.10}).  The embedding $H^s_G(\mathbb{S}^n,g_{\mathbb{S}^n}) \hookrightarrow L^p(\mathbb{S}^n,g_{\mathbb{S}^n})$ is compact for $1\leq p \leq 2n/(n-2s)$ because $k=\min_{\zeta\in \mathbb{S}^n}\dim O^{\zeta}_Z\geq 1$ and
\begin{equation*}
H^s_G(\mathbb{S}^n,g_{\mathbb{S}^n})\hookrightarrow H^s_Z(\mathbb{S}^n,g_{\mathbb{S}^n})\hookrightarrow L^{p}(\mathbb{S}^n,g_{\mathbb{S}^n}),
\end{equation*}
where the second embedding is compact (Example \ref{gex5.4} and Lemma \ref{gl4.3}). Moreover $X_G=H^s_G(\mathbb{S}^n,g_{\mathbb{S}^n})$ satisfies the assumptions of Lemma \ref{gl4.1}. Thus we obtain an unbounded sequence of points $ u_l $ of $J$, with $f(\cdot,t)=f(t)$, lie in $H^s_G(\mathbb{S}^n,g_{\mathbb{S}^n})$. They cannot be rotationally symmetric on $\mathbb{S}^n$ with respect to $\zeta_3$ because they satisfy
\begin{equation*}
u_l(\zeta_1,\zeta_2,\zeta_3) = -u_l(\tau (\zeta_1,\zeta_2,\zeta_3))= -u_l(\zeta_2,\zeta_1,\zeta_3) \text{ for every } \zeta \in \mathbb{S}^n.
\end{equation*}
\end{example}

\begin{example}
\label{gex5.6} In order to cover the case $n=4$, we consider the subgroup $Z=SO(3)\times SO(2)\subset O(5)$. Define $G= \langle Z \cup \{\tau\} \rangle \subset  O(5)$, where $\tau \in O(5)$ is given by 
\begin{equation*}
\tau(\zeta_1,\zeta_2,\zeta_3,\zeta_4,\zeta_5)=(-\zeta_1,-\zeta_2,-\zeta_3,\zeta_4,\zeta_5), ~\zeta_i\in \mathbb{R}, ~i=1,...,5.
\end{equation*}
 The subgroup $G$ was recently used by Biliotti and Siciliano \cite{BS} to find nonradial solutions of \eqref{ge5.61} for the case $N=5$, as well as for any nonlocal problems on $\mathbb{R}^5$. Following as in Example \ref{gex5.5},  we can define the action of $G$ on $H^s(\mathbb{S}^n,g_{\mathbb{S}^n})$ by \eqref{ge5.60} and thus obtain the compactness of $H^s_G(\mathbb{S}^n,g_{\mathbb{S}^n})\hookrightarrow L^{p}(\mathbb{S}^n,g_{\mathbb{S}^n})$. Moreover, since the functions $u_m(\zeta_1,\zeta_2,\zeta_3,\zeta_4,\zeta_5)=\zeta_1^m$ are in $H^s_G(\mathbb{S}^n,g_{\mathbb{S}^n})$ for each $m$ odd, then  $\dim H^s_G(\mathbb{S}^n,g_{\mathbb{S}^n})=+\infty$. Therefore, by Lemma \ref{gl4.1} we have obtained unbounded sequence $\{u_l\}_{l\in \mathbb{N}}$ of solutions of \eqref{g1.1} in $H^s(\mathbb{S}^n,g_{\mathbb{S}^n})$, with $f(\cdot,t)=f(t)$, such that
\begin{equation*}
u_l(\zeta_1,\zeta_2,\zeta_3,\zeta_4,\zeta_5) = -u_l(-\zeta_1,-\zeta_2,-\zeta_3,\zeta_4,\zeta_5).
\end{equation*}
\end{example}

\begin{remark}
The subgroups chosen in Examples \ref{gex5.5} and \ref{gex5.6} guarantee the existence of infinitely many non-radial sign-changing solutions for the equation \eqref{1.2} with $s\in(0,n/2)$ and $n\geq 3$. In fact, by Lemma \ref{l2.1} with the proper choice of stereographic projection, there exists an unbounded sequence $\{v_l\}_{l\in\mathbb{N}}$ in $D^{s,2}(\mathbb{R}^n)$ of non-radial solutions of \eqref{1.2}. However, as we saw in the introduction that any positive solution of \eqref{1.2} has the same energy in $D^{s,2}(\mathbb{R}^n)$. Therefore, the solutions $v_l$ change sign for $l$ large.
\end{remark}

 {\it Proof of Theorem \ref{gt1.5}.} By Lemma \ref{gl4.3}, the embedding $X_G:=H^s_G(\mathbb{S}^n,g)\hookrightarrow L^p(\mathbb{S}^n,g)$ is compact for $p\leq 2n/(n-2s)$. Therefore, we may apply Lemma \ref{gl4.1} to complete the proof of Theorem \ref{gt1.5}.
 \hfill $\square$
 
 {\it Proof of Corollary \ref{gc1.6}.} By Theorem \ref{gt1.5}, the equation \eqref{g1.1}, with $f(\cdot, t)=f(t) = |t|^{p-1}t - \lambda t$ and $\lambda\in (0,\lambda^*)$, has an unbounded sequence $\{u_l\}_{l\in\mathbb{N}}$ in $H^s(\mathbb{S}^n, g)$.  Unless, of subsequence, we can assume that $\{u_l\}_{l\in\mathbb{N}}$ is a positive sequence. Then, from Theorem \ref{gt2.1}, each $u_l$ is constan. Multiplying \eqref{g1.1} by $u_l$ and integrating, we have
 \begin{equation*}
 f(u_l)u_l [vol_g(\mathbb{S}^n)]=  \int_{\mathbb{S}^n}u_l(P^g_s - R^g_s) u_l d\upsilon_g =0,
 \end{equation*}
 and
 \begin{equation*}
 u_l^2\int_{\mathbb{S}^n}R^g_sd\upsilon_g  = \int_{\mathbb{S}^n}u_l(R^g_s u +f(u_l))d\upsilon_g = \|u_l\|^2_{s, g} \rightarrow \infty \text{ as } l\rightarrow \infty.
 \end{equation*}
Then $u_l \rightarrow \infty$ when $l\rightarrow \infty$. On the other hand, there exist three positive constants $a_1$, $a_2$ and $\mu>2$ such that
\begin{equation*}
tf(t)\geq \mu F(t)\geq a_1 t^{\mu}-a_2 \text{ for all } t>0 \text{ large}.
\end{equation*}
Then
\begin{equation*}
0=u_lf(u_l)\geq a_1 u_l^{\mu}-a_2 >0 \text{ for } l \text{ large },
\end{equation*}
which is a contradiction.
\hfill $\square$

{\it Proof of Corollary \ref{c1.1}.} It follows from \cite{ABR} and from the proof of Corollary \ref{gc1.6}.
\hfill $\square$

\section{Appendix}\label{s4}

\subsection{Sobolev type spaces and conformally invariant operators on the unit sphere}\

In this subsection, we recall some results for $P_s$ and Bessel potential spaces on spheres which can be found in \cite{Bra, CM, PS, Stri, Trie}.

The operators $P_s$ have eigenfunctions the spherical harmonics. Let $Y^{(k)}$ be a spherical harmonic of degree $k\geq 0$. Then for $s\in (0,n/2)$, we have
\begin{equation*}
P_s(Y^{(k)})=\frac{\Gamma(k+\frac{n}{2}+s)}{\Gamma(k+\frac{n}{2}-s)} Y^{(k)}.
\end{equation*}
If $g\in[\mathbb{S}^n]$ is a conformal metric on $\mathbb{S}^n$, then $(P^g_s)^{-1}$, acting in $L^2(\mathbb{S}^n,g)$, is compact, self-adjoint and positive \cite[Proposition 2.3]{CM}. Consequently, $P^g_s$ admits an unbounded sequence of eigenvalues depending on $g$. On the other hand, the existence of eigenvalues for the operator $P^g_s - R^g_s$ was shown in \cite[Appendix]{ABR} for the case $s\in(0,1)$. We believe that (\ref{g1.40}) holds for $s>1$.

Let $\Delta_{g_{\mathbb{S}^n}}$ be the Laplace-Beltrami operator on $(\mathbb{S}^n,g_{\mathbb{S}^n})$. For $s>0$ and $1<p<\infty$, the Bessel potential space $H^s_p(\mathbb{S}^n)$ is the set consisting of all  $u\in L^p(\mathbb{S}^n)$ such that $(1-\Delta)^{\frac{s}{2}}u\in L^p(\mathbb{S}^n)$. The norm in $H^s_p(\mathbb{S}^n)$ is defined by
\begin{equation*}
\|u\|_{H^s_p(\mathbb{S}^n)}:=\|(1-\Delta_{g_{\mathbb{S}^n}})^{\frac{s}{2}}u\|_{L^p(\mathbb{S}^n)}.
\end{equation*}
When $p=2$, then $H^s_2(\mathbb{S}^n)$ coincides with $H^s(\mathbb{S},g_{\mathbb{S}^n})$ and the norms $\|\cdot\|_{H^s_p(\mathbb{S}^n)}$ and $\|\cdot\|_{s,g_{\mathbb{S}^n}}$ are equivalents. Recall that $H^s(\mathbb{S}^n,g)$ denotes the closure of $C^{\infty}(\mathbb{S}^n)$ under the norm
\begin{equation*}
\|u\|^2_{s,g} := \int_{\mathbb{S}^n} uP^g_su~d\upsilon_g.
\end{equation*}  

\begin{theorem}
\label{gat1.1}\
\begin{enumerate}
\item[(i)] If $sp<n$, then the embedding $H^s_p(\mathbb{S}^n)\hookrightarrow L^{q}(\mathbb{S}^n)$ is continuous for $1\leq q \leq np/(n-sp)$ and compact for $q<np/(n-sp)$. 
\item[(ii)] If $s\in \mathbb{N}$, then $H^s_p(\mathbb{S}^n)=W^{s,p}(\mathbb{S}^n, g_{\mathbb{S}^n})$.
\item[(iii)] If $s=k+s_0$ with $k\in \mathbb{N}$ and $s_0\in (0,1)$, then the embedding $H^s(\mathbb{S},g) \hookrightarrow W^{k,\frac{2n}{n-2s_0}}(\mathbb{S}^n,g)$ is continuous.
\end{enumerate}
\end{theorem}
\begin{proof}
The proof of (i) and (ii) can be found in the reference cited at the beginning of this section. For the proof of (iii), we use (\ref{r1.1}) and Becker \cite{Bec} to obtain
\begin{equation*}
\begin{aligned}
\int_{\mathbb{S}^n}uP^g_s u~d\upsilon_g & \geq C\|(1-\Delta_{g_{\mathbb{S}^n}})^{\frac{s}{2}}(\varphi u)\|_{L^2(\mathbb{S}^n)}\\
  & =C\|(1-\Delta_{g_{\mathbb{S}^n}})^{\frac{s_0}{2}}(1-\Delta_{g_{\mathbb{S}^n}})^{\frac{k}{2}}(\varphi u)\|_{L^2(\mathbb{S}^n)}\\
  &\geq C\| (1-\Delta_{g_{\mathbb{S}^n}})^{\frac{k}{2}}(\varphi u)\|_{L^{\frac{2n}{n-2s_0}}(\mathbb{S}^n)}\\
  & \geq C \|\varphi u\|_{H^{k,\frac{2n}{n-2s_0}}(\mathbb{S}^n)}
\end{aligned}
\end{equation*}
for all $u\in C^{\infty}(\mathbb{S}^n)$. 
\end{proof}

\bibliographystyle{abbrv}

\bibliography{bibliography}

\end{document}